\documentclass[12pt]{amsart}
\usepackage{amsmath,amssymb}
\usepackage{graphicx,bm}
\usepackage{geometry}
\usepackage{hyperref}
\usepackage{amsthm}
\usepackage{verbatim}
\usepackage{slashed}
\usepackage{cases}
\usepackage{color}
\numberwithin{equation}{section}

\newtheorem{thm}{Theorem}[section]

\newtheorem{cor}[thm]{Corollary}
\newtheorem{lem}[thm]{Lemma}

\newtheorem{rmk}[thm]{Remark}

\newcommand{\pt}{\partial}

\begin{document}

\title[$\alpha$-Dirac-harmonic map flow]{Short-time existence of the $\alpha$-Dirac-harmonic map flow and applications}
\author{J\"urgen Jost, Jingyong Zhu}
\address{Max Planck Institute for Mathematics in the Sciences, Inselstrasse 22, 04103 Leipzig, Germany}
\email{jost@mis.mpg.de}
\address{Max Planck Institute for Mathematics in the Sciences, Inselstrasse 22, 04103 Leipzig, Germany}
\email{jizhu@mis.mpg.de}



\begin{abstract}
In this paper, we discuss the general existence theory of Dirac-harmonic maps from closed surfaces via the heat flow for $\alpha$-Dirac-harmonic maps and blow-up analysis. More precisely, given any initial map along which the Dirac operator has nontrivial minimal kernel, we first prove the short time existence of the heat flow for $\alpha$-Dirac-harmonic maps. The obstacle to the global existence is the singular time when the kernel of the Dirac operator no longer stays minimal along the flow. In this case, the kernel may not be continuous even if the map is smooth with respect to time. To overcome this issue, we use the analyticity of the target manifold to obtain the density of the maps along which the Dirac operator has minimal kernel in the homotopy class of the given initial map. Then, when  we arrive at the singular time, this density allows us to pick another map which has lower energy to restart the flow. Thus, we get a flow which may not be continuous at a set of isolated points. Furthermore, with the help of small energy regularity and blow-up analysis, we finally get the existence of nontrivial $\alpha$-Dirac-harmonic maps ($\alpha\geq1$) from closed surfaces. Moreover, if the target manifold does not admit any nontrivial harmonic sphere, then the map part stays in the same homotopy class as the given initial map. 
\end{abstract}
\maketitle


\section{Introduction}  

Motivated by the supersymmetric nonlinear sigma model from quantum field theory, see \cite{jost2009geometry}, Dirac-harmonic maps from spin Riemann surfaces into Riemannian manifolds were introduced in \cite{chen2006dirac}. They are generalizations of the classical harmonic maps and harmonic spinors. From the variational point of view, they are critical points of a conformal invariant  action functional whose {Euler-Lagrange equations are }a coupled elliptic system consisting of a second order equation and a Dirac equation.

It turns out that the existence of  Dirac-harmonic maps from closed surfaces is a very difficult problem. Different from the Dirichlet problem, even if there is no bubble, the nontriviality of the limit is also an issue. Here, a solution is considered  trivial if the spinor part $\psi$ vanishes identically. So far, there are only a few results about Dirac-harmonic maps from closed surfaces, see \cite{ammann2013dirac} and \cite{yang2009structure}\cite{chen2015dirac} for uncoupled Dirac-harmonic maps (here uncoupled means that the map part is harmonic; by an observation of Bernd Ammann and Johannes Wittmann, this is the typical case) based on index theory and the Riemann-Roch theorem, respectively. In an important contribution \cite{wittmann2017short}, Wittmann  investigated the heat flow  introduced in \cite{chen2017estimates} and showed the short-time existence of this flow; for reasons that will become apparent below this is not as easy as for other parabolic systems. The problem has also been approached by linking and Morse-Floer theory. See \cite{isobe2012existence}\cite{isobe2017morse} for one dimension and \cite{jost2019alpha} for the two dimensional case.  

In critical point theory, the Palais-Smale condition is a very strong and useful tool. It fails, however, for many of the basic problems in geometric analysis, and in particular for the energy functional of harmonic maps from spheres \cite{jost2017riemannian}. Therefore, it is not expected to be true for Dirac-harmonic maps. To overcome this problem for harmonic maps, Sacks-Uhlenbeck \cite{sacks1981existence} introduced the notion of $\alpha$-harmonic maps where the integrand in the energy functional is raised to a power $\alpha >1$. These $\alpha$-harmonic maps then satisfy the Palais-Smale condition. However, when we analogously introduce $\alpha$-Dirac-harmonic maps, the Palais-Smale condition fails due to the following existence result for uncoupled $\alpha$-Dirac-harmonic maps, which directly follows from the proof of Theorem \ref{Existence of alpha DH map}.
\begin{thm}
For a closed spin surface $M$ and a closed manifold $N$, consider a homotopy class $[\phi]$ of maps $\phi: M^m\to N^n$
for which   $[{\rm dim}_{\mathbb{H}}({\rm ker}\slashed{D}_\phi)]_{\mathbb{Z}_2}$ is non-trivial. Assume that $\phi_0\in[\phi]$ is an $\alpha$-harmonic map. Then there is a real vector space $V$ of real dimension $4$ such that all
$(\phi_0,\psi)$, $\psi\in V$, are $\alpha$-Dirac-harmonic maps.
\end{thm}

 To overcome this issue, in \cite{isobe2012existence}\cite{isobe2017morse}, the authors add an extra nonlinear term to the action functional of Dirac-geodesics. As for the two dimensional case \cite{jost2019alpha}, we even cannot directly prove the Palais-Smale condition for the action functional of perturbed Dirac-harmonic maps {into non-flat target manifolds}. Instead, we are only able to prove it for perturbed $\alpha$-Dirac-harmonic maps, and then approximate the $\alpha$-Dirac-harmonic map by a sequence of perturbed $\alpha$-Dirac-harmonic maps. However, in this approach, it is not easy to control the energies of the perturbed $\alpha$-Dirac-harmonic maps, which are constructed by a Min-Max method over  increasingly large domains in the configuration space.

Due to these two problems, {in this paper}, we would like to use the heat flow method to get the existence of Dirac-harmonic maps from closed surfaces to general manifolds where the harmonic map type equation is parabolized and the first order Dirac equation is carried along as an elliptic side constraint \cite{chen2017estimates}. As already mentioned, the short-time existence of the  heat flow for Dirac-harmonic map was proved by Wittmann \cite{wittmann2017short}. He constructed the solution to the constraint Dirac equation by the projector of the Dirac operator along maps. By assuming that the Dirac operator along the initial map has nontrivial minimal kernel, he showed that the kernel would stay minimal for small time in the homotopy class of the initial map. This minimality implies a uniform bound for the resolvents and the Lipschitz continuity of the normalized Dirac kernel along the flow. This Lipschitz continuity makes the Banach fixed point theorem available. If one follows this approach, the first issue is how to deal with the kernel jumping problem. Observe that if the Dirac operators converge at the jumping time, the symmetry of the spectrum of Dirac operator guarantees that the limiting Dirac operator has odd dimensional kernel. Therefore, it is natural to try to extend Wittmann's short time existence to the odd dimensional case. However, the eigenvalues in this case may split at time $t=0$. Then the projector may not be continuous even if the Dirac operator is smooth with respect to time along the flow (see \cite{kato2013perturbation}), which means that the Lipschitz continuity of the kernel is not available in general. To overcome this issue, we need the density mentioned in the abstract, which gives us a piecewise smooth flow.

As for the convergence, it is sufficient to control the energy of the spinor field because the energy of the map decreases along the flow. To do so, one can impose a restriction on the energy of the initial map as in \cite{jost2017global} and get the existence of Dirac-harmonic maps when the initial map has small energy. Alternatively, we use another type flow, that is, the heat flow  for $\alpha$-Dirac-harmonic maps (also called $\alpha$-Dirac-harmonic map flow in the literatures). Our motivation comes from the successful application of this flow to the Dirichlet problem \cite{jost2018geometric}. Different from there, we cannot uniquely solve the constraint equation. Moreover, our equations of the flow are different. We never write the constraint equation in the Euclidean space $\mathbb{R}^q$. Instead, we just solve it in the target manifold $N$. Last, our flow is not unique due to the absent of a boundary. Instead, only a weak uniqueness is available. Consequently, we need prove the fact that the flow takes value in the target manifold $N$ in a different way. Eventually, we shall obtain the  following results on the general existence of Dirac-harmonic maps.

\begin{thm}\label{existence of DH maps}
Let $M$ be a closed spin surface and $(N,h)$ a real analytic closed manifold. Suppose there exists a map $u_0\in C^{2+\mu}(M,N)$ for some $\mu\in(0,1)$ such that ${\rm dim}_{\mathbb{H}}{\rm ker}\slashed{D}^{u_0}=1$. Then there exists a nontrivial smooth Dirac-harmonic map $(\Phi,\Psi)$ {satisfying $E(\Phi)\leq E(u_0)$ and $\|\Psi\|_{L^2}=1$. }

 Furthermore, if $(N,h)$ does not admit any nontrivial harmonic sphere, then the map part $\Phi$ is in the same homotopy class as $u_0$ and $(\Phi,\Psi)$ is coupled if the energy of the map is strictly bigger than the energy minimizer in the homotopy class $[u_0]$. 
\end{thm}

\begin{rmk}
The analyticity of the target manifold is a sufficient condition which is used to get the density mentioned in the abstract. In fact, it is easy to see from the proof that we only need the density of the following set
\begin{equation}
\begin{split}
Y:=\{e\in(m^{\alpha_i}_0,+\infty)| &\text{there exists at least one map} \ u \ \text{such that}\\
&\quad {\rm dim}_{\mathbb{H}}{\rm ker}\slashed{D}^{u}=1 \ \text{and} \ E^{\alpha_i}(u)=e \}
\end{split}
\end{equation}
 at the $\alpha_i$-energy minimizer $m^{\alpha_i}_0$ in the homotopy class $[u_0]$ for a sequence $\alpha_i\searrow1$ as $i\to\infty$.

\end{rmk}

In \cite{wittmann2019minimal}, Wittmann discussed the density of those maps along which all the Dirac operators have minimal kernel. In particular, we have the following corollary.
\begin{cor}
Let $M$ be a closed spin surface and $(N,h)$ a real analytic closed manifold. We also assume that 

(1) M is connected, oriented and of positive genus;

(2) N is connected. If $N$ is even-dimensional, then we assume that it is non-orientable.

Then there exists a nontrivial smooth Dirac-harmonic map.
\end{cor}

The rest of paper is organized as follows: In Section 2, we recall some definitions, notations and lemmas  about Dirac-harmonic maps and the kernel of Dirac operator. In Section 3, under the minimality assumption on the kernel of the Dirac operator along the initial map, we prove the short time existence, weak uniqueness and regularity of the heat flow for $\alpha$-Dirac-harmonic maps. In Section 4, we prove the existence of $\alpha$-Dirac-harmonic maps and Theorem \ref{existence of DH maps}. In the Appendix,  we solve the constraint equation and prove Lipschitz continuity of the solution with respect to the map.

\section{Preliminaries}    

Let $(M, g)$ be a compact surface with a fixed spin structure. On the spinor bundle $\Sigma M$, we denote the Hermitian inner product by $\langle\cdot, \cdot\rangle_{\Sigma M}$. For any $X\in\Gamma(TM)$ and $\xi\in\Gamma(\Sigma M)$, the Clifford multiplication satisfies the following skew-adjointness:
\begin{equation}
\langle X\cdot\xi, \eta\rangle_{\Sigma M}=-\langle\xi, X\cdot\eta\rangle_{\Sigma M}.
\end{equation}
Let $\nabla$ be the Levi-Civita connection on $(M,g)$. There is a connection (also denoted by $\nabla$) on $\Sigma M$ compatible with $\langle\cdot, \cdot\rangle_{\Sigma M}$.  Choosing a local orthonormal basis $\{e_{\beta}\}_{\beta=1,2}$ on $M$, the usual Dirac operator is defined as $\slashed\partial:=e_\beta\cdot\nabla_\beta$, where $\beta=1,2$. Here and in the sequel, we use the Einstein summation convention. One can find more about spin geometry in \cite{lawson1989spin}.

Let $\phi$ be a smooth map from $M$ to another compact Riemannian manifold $(N, h)$ of dimension $n\geq2$. Let $\phi^*TN$ be the pull-back bundle of $TN$ by $\phi$ and consider the twisted bundle $\Sigma M\otimes\phi^*TN$. On this bundle there is a metric $\langle\cdot,\cdot\rangle_{\Sigma M\otimes\phi^*TN}$ induced from the metric on $\Sigma M$ and $\phi^*TN$. Also, we have a connection $\tilde\nabla$ on this twisted bundle naturally induced from those on $\Sigma M$ and $\phi^*TN$. In local coordinates $\{y^i\}_{i=1,\dots,n}$, the section $\psi$ of $\Sigma M\otimes\phi^*TN$ is written as 
$$\psi=\psi_i\otimes\partial_{y^i}(\phi),$$
where each $\psi^i$ is a usual spinor on $M$. We also have the following local expression of $\tilde\nabla$
$$\tilde\nabla\psi=\nabla\psi^i\otimes\partial_{y^i}(\phi)+\Gamma_{jk}^i(\phi)\nabla\phi^j\psi^k\otimes\partial_{y^i}(\phi),$$
where $\Gamma^i_{jk}$ are the Christoffel symbols of the Levi-Civita connection of $N$. The Dirac operator along the map $\phi$ is defined as
\begin{equation}\label{dirac}
\slashed{D}:=e_\alpha\cdot\tilde\nabla_{e_\alpha}\psi=\slashed\partial\psi^i\otimes\partial_{y^i}(\phi)+\Gamma_{jk}^i(\phi)\nabla_{e_\alpha}\phi^j(e_\alpha\cdot\psi^k)\otimes\partial_{y^i}(\phi),
\end{equation}
which is self-adjoint \cite{jost2017riemannian}. Sometimes, we use $\slashed{D}_\phi$ to distinguish the Dirac operators defined on different maps. In \cite{chen2006dirac}, the authors  introduced the  functional
\begin{equation}\begin{split}
L(\phi,\psi)&:=\frac12\int_M(|d\phi|^2+\langle\psi,\slashed{D}\psi\rangle_{\Sigma M\otimes\phi^*TN})\\
&=\frac12\int_M h_{ij}(\phi)g^{\alpha\beta}\frac{\partial\phi^i}{\partial x^\alpha}\frac{\pt\phi^j}{\pt x^\beta}+h_{ij}(\phi)\langle\psi^i,\slashed{D}\psi^j\rangle_{\Sigma M}.
\end{split}
\end{equation}
They computed the Euler-Lagrange equations of $L$:
\begin{equation}\label{eldh1}
\tau^m(\phi)-\frac12R^m_{lij}\langle\psi^i,\nabla\phi^l\cdot\psi^j\rangle_{\Sigma M}=0,
\end{equation}
\begin{equation}\label{eldh2}
\slashed{D}\psi^i=\slashed\pt\psi^i+\Gamma_{jk}^i(\phi)\nabla_{e_\alpha}\phi^j(e_\alpha\cdot\psi^k)=0,
\end{equation}
where $\tau^m(\phi)$ is the $m$-th component of the tension field \cite{jost2017riemannian} of the map $\phi$ with respect to the coordinates on $N$, $\nabla\phi^l\cdot\psi^j$ denotes the Clifford multiplication of the vector field $\nabla\phi^l$ with the spinor $\psi^j$, and $R^m_{lij}$ stands for the component of the Riemann curvature tensor of the target manifold $N$. Denote 
$$\mathcal{R}(\phi,\psi):=\frac12R^m_{lij}\langle\psi^i,\nabla\phi^l\cdot\psi^j\rangle_{\Sigma M}\pt_{y^m}.$$ We can write \eqref{eldh1} and \eqref{eldh2} in the following global form:
\begin{numcases}{}
\tau(\phi)=\mathcal{R}(\phi,\psi), \label{geldh1} \\
\slashed{D}\psi=0,  \label{geldh2}
\end{numcases}
and call the solutions $(\phi,\psi)$ Dirac-harmonic maps from $M$ to $N$.

With the aim to get a general existence scheme for  Dirac-harmonic maps, the following heat flow for Dirac-harmonic maps was introduced in \cite{chen2017estimates}:
\begin{numcases}{}
		\partial_tu=\tau(u)-\mathcal{R}(u,\psi), & on $(0,T)\times M$, \label{map}\\
		\slashed{D}^{u}\psi=0, & on $[0,T]\times M$.	\label{dirac}
	\end{numcases}
When $M$ has boundary, the short time existence and uniqueness of \eqref{map}-\eqref{dirac} was also shown in \cite{chen2017estimates}. Furthermore, the existence of a global weak solution to this flow in dimension two under some boundary-initial constraint was obtained in \cite{jost2017global}. In \cite{jost2018geometric}, to remove the restriction on the initial maps, the authors refined an estimate about the spinor in \cite{chen2017estimates} as follows:
 \begin{lem}\label{spinor norm with boundary}\cite{jost2018geometric}
 Let $M$ be a compact spin Riemann surface with boundary $\partial M$, N be a compact Riemann manifold. Let $u\in W^{1,2\alpha}(M,N)$ for some $\alpha>1$ and $\psi\in W^{1,p}(M, \Sigma M\otimes u^*TN)$ for $1<p<2$, then there exists a positive constant $C=C(p,M,N,\|\nabla u\|_{L^{2\alpha}})$ such that
 \begin{equation}
 \|\psi\|_{W^{1,p}(M)}\leq C(\|\slashed{D}\psi\|_{L^{p}(M)}+\|\mathcal{B}\psi\|_{W^{1-1/p,p}(\partial M)}).
 \end{equation}
 \end{lem}
 Motivated by this lemma, they considered the  $\alpha$-Dirac-harmonic flow and got the existence of Dirac-harmonic maps. For a closed manifold $M$, the situation is much more complicated because the kernel of the Dirac operator is a linear space. If the Dirac operator along the initial map has one dimensional kernel, Wittmann proved the short time existence on $M$ whose dimension is $m\equiv0,1,2,4({\rm mod} \ 8)$.

By \cite{nash1956imbedding}, we can isometrically embed $N$ into $\mathbb{R}^q$. Then \eqref{geldh1}-\eqref{geldh2} is equivalent to following system:
\begin{numcases}{}
		\Delta_g{u}=II(du,du)+Re(P(\mathcal{S}(du(e_\beta),e_{\beta}\cdot\psi);\psi)), \label{Emap }\\
		\slashed{\partial}\psi=\mathcal{S}(du(e_\beta),e_{\beta}\cdot\psi),	\label{Edirac}
	\end{numcases}
where $II$ is the second fundamental form of $N$ in $\mathbb{R}^q$, and 
\begin{equation}
\mathcal{S}(du(e_\beta),e_{\beta}\cdot\psi):=(\nabla{u^A}\cdot\psi^B)\otimes II(\partial_{z^A},\partial_{z^B}),
\end{equation}
\begin{equation}
Re(P(\mathcal{S}(du(e_\beta),e_{\beta}\cdot\psi);\psi)):=P(S(\partial_{z^C},\partial_{z^B});\partial_{z^A})Re(\langle\psi^A,du^C\cdot\psi^B\rangle).
\end{equation}
Here $P(\xi;\cdot)$ denotes the shape operator, defined by $\langle P(\xi;X),Y\rangle=\langle A(X,Y),\xi\rangle$ for $X,Y\in\Gamma(TN)$ and $Re(z)$ denotes the real part of $z\in\mathbb{C}$. Together with the \emph{nearest point projection}:
\begin{equation}
\pi: \ N_{\delta}\to N,
\end{equation}
where $N_{\delta}:=\{z\in\mathbb{R}^q| d(z,N)\leq\delta\}$, we can rewrite the evolution equation \eqref{map} as an equation in $\mathbb{R}^q$.

\begin{lem}\cite{wittmann2017short}\cite{chen2017estimates}\label{eq in Euclidean}
A tuple $(u,\psi)$, where $u:[0,T]\times M\to N$ and $\psi\in \Gamma(\Sigma M\otimes u^*TN)$, is a solution of \eqref{map} if and only if 
\begin{equation}
\partial_tu^A-\Delta u^A=-\pi^A_{BC}(u)\langle\nabla u^B,\nabla u^C\rangle-\pi^A_B(u)\pi^C_{BD}(u)\pi^C_{EF}(\psi^D,\nabla u^E\cdot\psi^F)
\end{equation}
on $(0,T)\times M$, for $A=1,\dots,q$. Here we denote the $A$-th component function of $u: [0,T]\times M\to N\subset\mathbb{R}^q$ by $u^A: M\to\mathbb{R}$, write $\pi^A_B(z)$ for the $B$-th partial derivative of the $A$-th component function of $\pi: \mathbb{R}^q\to\mathbb{R}^q$ and the global sections $\psi^A\in\Gamma(\Sigma M)$ are defined by $\psi=\psi^A\otimes(\partial_A\circ u)$, where $(\partial_A)_{A=1,\dots,q}$ is the standard basis of $T\mathbb{R}^q$. Moreover, $\nabla$ and $\langle\cdot,\cdot\rangle$ denote the gradient and the Riemannian metric on $M$, respectively.  
\end{lem}

For future reference, we define
\begin{equation}\label{F1A}
F_1^A(u):=-\pi^A_{BC}(u)\langle\nabla u^B,\nabla u^C\rangle,
\end{equation}
\begin{equation}\label{F2A}
F_2^A(u,\psi):=-\pi^A_B(u)\pi^C_{BD}(u)\pi^C_{EF}(\psi^D,\nabla u^E\cdot\psi^F).
\end{equation}
Note that for $u\in C^1(M,N)$ and $\psi\in\Gamma(\Sigma M\otimes u^*TN)$ we have
\begin{equation}
II(du_p(e_\alpha), du_p(e_\alpha)))=-F_1^A(u)|_p\partial_A|_{u(p)},
\end{equation}
\begin{equation}
\mathcal{R}(\phi,\psi)|_p=-F_2^A(u,\psi)|_p\partial_A|_{u(p)}
\end{equation}
for all $p\in M$, where $\{e_\alpha\}$ is an orthonormal basis of $T_pM$. 

Next, for every $T>0$, we denote by $X_T$ the Banach space of bounded maps:
\begin{equation}
X_T:=B([0,T]; C^1(M,\mathbb{R}^q)),
\end{equation}
\begin{equation}
\|u\|_{X_T}:=\max\limits_{A=1,\dots,q}\sup\limits_{t\in[0,T]}(\|u^A(t,\cdot)\|_{C^0(M)}+\|\nabla u^A(t,\cdot)\|_{C^0(M)}).
\end{equation}
For any map $v\in X_T$, the closed ball with center $v$ and radius $R$ in $X_T$ is defined by
\begin{equation}
B_R^T(v):=\{u\in X_T|\|u-v\|\leq R\}.
\end{equation}
We denote by $P^{u_t,v_s}=P^{u_t,v_s}(x)$ the parallel transport of $N$ along the unique shortest geodesic from $\pi(u(x,t))$ to $\pi(v(x,s))$. We also denote by $P^{u_t,v_s}$ the inducing  mappings
\begin{equation}
(\pi\circ u_t)^*TN\to(\pi\circ v_s)^*TN,
\end{equation}
\begin{equation}
\Sigma M\otimes(\pi\circ u_t)^*TN\to\Sigma M\otimes(\pi\circ v_s)^*TN
\end{equation}
and 
\begin{equation}
\Gamma_{C^1}(\Sigma M\otimes(\pi\circ u_t)^*TN)\to\Gamma_{C^1}(\Sigma M\otimes(\pi\circ v_s)^*TN).
\end{equation}

Now, let us define 
\begin{equation}
\Lambda(u_t)=\sup\{\tilde\Lambda| {\rm spec}(\slashed{D}^{\pi\circ u_t})\setminus\{0\}\subset\mathbb{R}\setminus(-\tilde\Lambda(u_t),\tilde\Lambda(u_t))\}
\end{equation}
 and $\gamma_t(x): [0,2\pi]\to\mathbb{C}$ as
 \begin{equation}
 \gamma_t(x):=\frac{\Lambda(u_t)}{2}e^{ix}.
 \end{equation}
 In general, we also denote by $\gamma$ the curve $\gamma(x): [0,2\pi]\to\mathbb{C}$ as
 \begin{equation}\label{general curve}
 \gamma(x):=\frac{\Lambda}{2}e^{ix}
 \end{equation}
 for some constant $\Lambda$ to be determined. Then the orthogonal projection onto ${\rm ker}(\slashed{D}^{\pi\circ u_t})$, which is the mapping \begin{equation}
\Gamma_{L^2}(\Sigma M\otimes(\pi\circ u_t)^*TN)\to\Gamma_{L^2}(\Sigma M\otimes(\pi\circ u_t)^*TN),
\end{equation}
can be written by the resolvent by
\begin{equation}
s\mapsto-\frac{1}{2\pi i}\int_{\gamma_t}R(\lambda,\slashed{D}^{\pi\circ u_t})sd\lambda,
\end{equation}
where $R(\lambda,\slashed{D}^{\pi\circ u_t}): \Gamma_{L^2}\to\Gamma_{L^2}$ is the resolvent of $\slashed{D}^{\pi\circ u_t}: \Gamma_{W^{1,2}}\to\Gamma_{L^2}$.

Finally, the following density lemma is very useful for us to extend the flow beyond the singular time.

\begin{lem}\label{density}\cite{wittmann2019minimal}
Let $M$ be a closed spin surface and $(N,h)$ a real analytic closed manifold. Suppose there exists a map $u_0\in C^{2+\mu}(M,N)$ for some $\mu\in(0,1)$ such that ${\rm dim}_{\mathbb{H}}{\rm ker}\slashed{D}^{u_0}=1$. Then  the kernel of $\slashed{D}^{u}$ is minimal for generic $u\in[u_0]$, i.e., for a $C^\infty$-dense and $C^1$-open subset of $[u_0]$.
\end{lem}

\section{The heat flow for $\alpha$-Dirac-harmonic maps}
In this section, we will prove the short-time existence of the heat flow for $\alpha$-Dirac-harmonic maps. Since we are working on a closed surface $M$, we cannot uniquely solve the Dirac equation in the following system:
\begin{numcases}{}
		\partial_tu=\frac{1}{(1+|\nabla u|^2)^{\alpha-1}}\bigg(\tau^\alpha(u)-\frac{1}{\alpha}\mathcal{R}(u,\psi)\bigg), \label{alpha map}\\
		\slashed{D}^{u}\psi=0.	\label{alpha dirac}
	\end{numcases}
 The short time existence and its extension are the obstacles. This system (if it converges) leads to a $\alpha$-Dirac-harmonic map which is a solution of the system
 \begin{equation}\label{alpha geldh}
\begin{split}
\tau^\alpha(u)&:=\tau((1+|du|^2)^\alpha)=\frac{1}{\alpha}\mathcal{R}(u,\psi)\\
\slashed{D}^u\psi&=0.
\end{split}
\end{equation}
and equivalently a critical point of  functional 
 \begin{equation}
 L^\alpha(u,\psi)=\frac12\int_M(1+|du|^2)^\alpha+\frac12\int_M\langle\psi,\slashed{D}^u\psi\rangle_{\Sigma M\otimes\phi^*TN},
 \end{equation}
 where $\tau$ is the tension field.

\subsection{Short time existence}

As in Section 2, we now embed $N$ into $\mathbb{R}^q$. Let $u: M\to N$ with $u=(u^A)$ and denote the spinor along the map $u$ by $\psi=\psi^A\otimes(\partial_A\circ u)$, where $\psi^A$ are spinors over $M$. For any smooth map $\eta\in C^\infty_0(M,\mathbb{R}^q)$ and any smooth spinor field $\xi\in C^\infty_0(\Sigma M\otimes\mathbb{R}^q)$, we consider the variation 
\begin{equation}\label{var}
u_t=\pi(u+t\eta), \ \ \ \psi^A_t=\pi^A_B(u_t)(\psi^B+t\xi^B),
\end{equation}
where $\pi$ is the nearest point projection as in Section 2. Then we have
\begin{lem}\label{alpha EL}
The Euler-Lagrange equations for $L^\alpha$ are 
\begin{equation}
\begin{split}
\Delta u^A&=-2(\alpha-1)\frac{\nabla^2_{\beta\gamma}u^B\nabla_{\beta}u^B\nabla_{\gamma}u^A}{1+|\nabla u|^2}+\pi^A_{BC}(u)\langle\nabla{u^B},\nabla{u^C}\rangle\\
&\quad+\frac{\pi^A_B(u)\pi^C_{BD}(u)\pi^C_{EF}(u)\langle\psi^D,\nabla{u}^E\cdot\psi^F\rangle}{\alpha(1+|\nabla{u}|^2)^{\alpha-1}}
\end{split}\end{equation}
and 
\begin{equation}
\slashed{\partial}\psi^A=\pi^A_{BC}(u)\nabla{u}^B\cdot\psi^C.
\end{equation}
\end{lem}
\begin{proof}
Suppose $(u,\psi)$ is a critical point of $L^\alpha$, then for the variation \eqref{var} we have
\begin{equation}
\begin{split}
\frac{dL^\alpha(u_t,\psi_t)}{dt}|_{t=0}&=\alpha\int_M(1+|\nabla{u}|^2)^{\alpha-1}\langle\nabla{u^A},\pi^A_B\nabla\eta^B+\pi^A_{BC}\nabla{u^C}\eta^B\rangle\\
&\quad+\int_M\langle\slashed\partial\psi^A,\pi^A_B\xi^B+\pi^A_{BC}\pi^C_D\psi^B\eta^D\rangle,\\
&=:I+II.
\end{split}\end{equation}
Then the lemma directly follows from the following computations.
\begin{equation*}
\begin{split}
I&=\alpha\int_M(1+|\nabla{u}|^2)^{\alpha-1}\langle\nabla{u^A},\nabla\eta^A\rangle+\alpha\int_M(1+|\nabla{u}|^2)^{\alpha-1}\pi^A_{BC}\langle\nabla{u^B},\nabla{u^C}\rangle\eta^A\\
&=-\alpha\int_M(1+|\nabla{u}|^2)^{\alpha-1}\Delta{u^A}\eta^A-\alpha(\alpha-1)\int_M(1+|\nabla{u}|^2)^{\alpha-2}\langle\nabla|\nabla{u}|^2,\nabla{u^A}\rangle\eta^A\\
&=-\alpha\int_M(1+|\nabla{u}|^2)^{\alpha-1}\bigg(\Delta{u^A}+2(\alpha-1)\frac{\nabla^2_{\beta\gamma}u^B\nabla_{\beta}u^B\nabla_{\gamma}u^A}{1+|\nabla u|^2}-\pi^A_{BC}(u)\langle\nabla{u^B},\nabla{u^C}\rangle\bigg)\eta^A.
\end{split}
\end{equation*}

\begin{equation*}
\begin{split}
II&=\int_M\langle\slashed\partial\psi^A-\pi^A_{BC}\nabla{u^B}\cdot\psi^C,\xi^A\rangle+\int_M\pi^A_B\pi^C_{BD}\langle\psi^D,\slashed\partial\psi^C\rangle\eta^A\\
&=\int_M\langle\slashed\partial\psi^A-\pi^A_{BC}\nabla{u^B}\cdot\psi^C,\xi^A\rangle+\int_M\pi^A_B\pi^C_{BD}\langle\psi^D,\slashed\partial\psi^C-\pi^C_{EF}\nabla{u^E}\cdot\psi^F\rangle\eta^A\\
&+\int_M\pi^A_B\pi^C_{BD}\langle\psi^D,\pi^C_{EF}\nabla{u^E}\cdot\psi^F\rangle\eta^A.
\end{split}
\end{equation*}
\end{proof}

 Lemma \ref{alpha EL} implies that \eqref{alpha map}-\eqref{alpha dirac} is equivalent to 
\begin{numcases}{}
\begin{split}
		\partial_tu^A&=\Delta{u^A}+2(\alpha-1)\frac{\nabla^2_{\beta\gamma}u^B\nabla_{\beta}u^B\nabla_{\gamma}u^A}{1+|\nabla u|^2}-\pi^A_{BC}(u)\langle\nabla{u^B},\nabla{u^C}\rangle \label{system in Euclidean1}\\
		&\quad-\frac{\pi^A_B(u)\pi^C_{BD}(u)\pi^C_{EF}(u)\langle\psi^D,\nabla{u}^E\cdot\psi^F\rangle}{\alpha(1+|\nabla{u}|^2)^{\alpha-1}}
		\end{split}\\
		\slashed{D}^{\pi\circ u}\psi=0,	\label{system in Euclidean2}
	\end{numcases}

Now, let us state the main result of this subsection.

\begin{thm}\label{short-time alpha flow}
Let $M$ be a closed surface, and $N$ a closed $n$-dimensional Riemannian manifold. Let $u_0\in C^{2+\mu}(M,N)$ for some $0<\mu<1$ with ${\rm dim}_{\mathbb H}{\rm ker}(\slashed{D}^{u_0})=1$ and $\psi_0\in{\rm ker}(\slashed{D}^{u_0})$ with $\|\psi_0\|_{L^2}=1$. Then there exists $\epsilon_1=\epsilon_1(M,N)>0$ such that, for any $\alpha\in(1,1+\epsilon_1)$, the problem \eqref{alpha map}-\eqref{alpha dirac} has a solution $(u,\psi)$ with 
\begin{equation}\label{initial value}
		\begin{cases}
		 \|\psi_t\|_{L^2}=1, & \ \forall t\in[0,T],\\
		u|_{t=0}=u_0, \  \psi|_{t=0}=\psi_0.
		\end{cases}
	\end{equation}
 satisfying
\begin{equation}
u\in C^{2+\mu,1+\mu/2}(M\times[0,T],N)
\end{equation}
and
\begin{equation}
\psi\in C^{\mu,\mu/2}(M\times[0,T], \Sigma M\otimes u^*TN)\cap L^\infty([0,T];C^{1+\mu}(M)).
\end{equation}
for some $T>0$. 
\end{thm}

\begin{proof}
{\bf Step 1:} Solving \eqref{system in Euclidean1}-\eqref{system in Euclidean2} in $\mathbb{R}^q$.

  In this step, we want to find a solution $u: M\times[0,T]\to\mathbb{R}^q$ and $\psi_t: M\to\Sigma M\otimes(\pi\circ{u_t})^*TN$ of \eqref{system in Euclidean1}-\eqref{system in Euclidean2} with the initial values \eqref{initial value}.  We first give a solution to \eqref{system in Euclidean2} in a neighborhood of $u_0$. For any $T>0$, we can choose $\epsilon$, $\delta$ and $R$ as in the Appendix such that
\begin{equation}\label{property1'}
u(x,t)\in N_\delta
\end{equation}
and 
\begin{equation}\label{property2'}
d^N((\pi\circ u)(x,t),(\pi\circ v)(x,s))<\epsilon<\frac12{\rm inj}(N)
\end{equation}
for all $u,v\in B^T_R:=B^T_R(\bar{u}_0)=\{u\in X_T|\|u-\bar{u}_0\|_{X_T}\leq R\}\cap\{u|_{t=0}=u_0\}$, $x\in M$ and $t,s\in[0,T]$, where $\bar{u}_0(x,t)=u_0(x)$ for any $t\in[0,T]$. If $R$ is small enough, then by Lemma \ref{dim of kernel u}, we have  
\begin{equation}
{\rm dim}_{\mathbb{K}}{\rm ker}(\slashed{D}^{\pi\circ u_t})=1
\end{equation}
and there exists $\Lambda=\frac12\Lambda(u_0)$ such that 
\begin{equation}
\#\{{\rm spec}(\slashed{D}^{\pi\circ u_t})\cap[-\Lambda,\Lambda]\}=1
\end{equation}
for any $u\in B_R^T$ and $t\in[0,T]$, where $\Lambda(u_0)$ is a constant such that ${\rm spec}(\slashed{D}^{u_0})\setminus\{0\}\subset\mathbb{R}\setminus[-\Lambda(u_0),\Lambda(u_0)]$. Furthermore, for $\psi_0\in{\rm ker}(\slashed{D}^{u_0})$ with $\|\psi_0\|_{L^2}=1$, Lemma \ref{projection to kernel} implies that
\begin{equation}
\sqrt{\frac34}\leq\|\tilde\psi_1^{u_t}\|_{L^2}\leq1
\end{equation} 
for any $u\in B_{R_1}^{T}$ and $t\in[0,T]$, where $\tilde\psi^{u_t}=P^{u_0,u_t}\psi=\tilde\psi^{u_t}_1+\tilde\psi^{u_t}_2$ with respect to the decomposition $\Gamma_{L^2}={\rm ker}(\slashed{D}^{\pi\circ u_t})\oplus({\rm ker}(\slashed{D}^{\pi\circ u_t}))^\bot$ and $R_1=R_1(R,\epsilon,u_0)>0$. 

Now, for any $T>0$ and $\kappa>0$, we define
$$V^T_{\kappa}:=\{v\in C^{1+\mu,\frac{1+\mu}{2}}(M\times[0,T])|\|v\|_{C^{1+\mu,\frac{1+\mu}{2}}}\leq\kappa, \ v|_{M\times\{0\}}=0\}.$$
Then, there exists $\kappa_{R_1}:=\kappa(R_1)>0$ such that 
\begin{equation}
u_0+v\in B^{T}_{R_1}, \ \forall v\in V_{\kappa}^T, \ \forall\kappa\leq\kappa_{R_1}.
\end{equation}
Now, we denote $\kappa_0:=\kappa_{R_1}$ and $V^T:= V^T_{\kappa_0}$.

For every $v\in V^T$, $u_0+v\in B^{T}_{R_1}$, Lemma \ref{lip} gives us a solution $\psi(v+u_0)$ to the  constraint equation. Since $v+u_0\in C^{1+\mu}(M)$, by $L^p$ regularity \cite{wittmann2017short} and Schauder estimate \cite{chen2017estimates}, we have 
\begin{equation}\label{C1+mu}
\|\psi(v+u_0)\|_{C^{1+\mu}(M)}\leq C(\mu, M, N, \kappa_0, \|u_0\|_{C^{1+\mu}(M)}).
\end{equation}
For any $0<t,s<T$, we also have
\begin{equation*}
\begin{split}
&\quad\slashed\partial(\psi(v+u_0)(t)-\psi(v+u_0)(s))\\
&=-\Gamma(\pi\circ(v+u_0)(t))\#\nabla(\pi\circ(v+u_0)(t))\#\psi(v+u_0)(t)\\
&\quad+\Gamma(\pi\circ(v+u_0)(s))\#\nabla(\pi\circ(v+u_0)(s))\#\psi(v+u_0)(s)\\
&=-\Gamma(\pi\circ(v+u_0)(t))\#\nabla(\pi\circ(v+u_0)(t))\#(\psi^v(t)-\psi(v+u_0)(s))\\
&\quad-\Gamma(\pi\circ(v+u_0)(t))\#(\nabla(\pi\circ(v+u_0)(t))-\nabla(\pi\circ(v+u_0)(s)))\#\psi(v+u_0)(t)\\
&\quad-(\Gamma(\pi\circ(v+u_0)(t))-\Gamma(\pi\circ(v+u_0)(s)))\#\nabla(\pi\circ(v+u_0)(s))\#\psi(v+u_0)(s),
\end{split}\end{equation*}
that is,
\begin{equation*}
\begin{split}
&\quad\slashed{D}^{\pi\circ v(t)}(\psi(v+u_0)(t)-\psi(v+u_0)(s))\\
&=-\Gamma(\pi\circ(v+u_0)(t))\#(\nabla(\pi\circ(v+u_0)(t))-\nabla(\pi\circ(v+u_0)(s)))\#\psi(v+u_0)(t)\\
&\quad-(\Gamma(\pi\circ(v+u_0)(t))-\Gamma(\pi\circ(v+u_0)(s)))\#\nabla(\pi\circ(v+u_0)(s))\#\psi(v+u_0)(s),
\end{split}\end{equation*}
where $\#$ denotes a multi-linear map with smooth coefficients. For any $\lambda\in(0,1)$, by the Sobolev embedding, $L^p$-regularity in \cite{wittmann2017short} and Lemma \ref{lip}, we have
\begin{equation}
\begin{split}
&\quad\|\psi(v+u_0)(t)-\psi(v+u_0)(s)\|_{C^\lambda(M)}\\
&\leq C(\lambda, M, N, \kappa_0, \|u_0\|_{C^1(M)})(\|v(t)-v(s)\|_{L^\infty(M)}+\|dv(t)-dv(s\|_{L^\infty}))\\
&\leq C(\lambda, M, N, \kappa_0, \|u_0\|_{C^1(M)})|t-s|^{\mu/2}.
\end{split}\end{equation}
Therefore, 
\begin{equation}\label{Cmu}
\|\psi(v+u_0)\|_{C^{\mu,\mu/2}(M)}\leq C(\mu, M, N, \kappa_0, \|u_0\|_{C^1(M)}).
\end{equation}

Now, when $\alpha-1$ is sufficiently small, for the $(v,\psi^v)$ above, the standard theory of linear parabolic systems (see \cite{schlag1996schauder}) implies that there exists a unique solution $v_1\in C^{2+\mu,1+\mu/2}(M\times[0,T], \mathbb{R}^q)$ to the following Dirichlet  problem:
\begin{numcases}{}
		\begin{split}
		\partial_t w^A&=\Delta_g w^A+2(\alpha-1)\frac{\nabla^2_{\beta\gamma}w^B\nabla_{\beta}(v+u_0)^B\nabla_{\gamma}(v+u_0)^A}{1+|\nabla(v+u_0)|^2}  \label{D1}\\
		&\quad+\pi^A_{BC}(v+u_0)\langle\nabla{(v+u_0)^B},\nabla{(v+u_0)^C}\rangle \ \\
		&\quad+\frac{(\pi^A_B\pi^C_{BD}\pi^C_{EF})(v+u_0)\langle\psi^D(v+u_0),\nabla{(v+u_0)}^E\cdot\psi^F(v+u_0)\rangle}{\alpha(1+|\nabla{(v+u_0)}|^2)^{\alpha-1}}, \\
		&\quad+\Delta_gu_0^A+2(\alpha-1)\frac{\nabla^2_{\beta\gamma}u_0^B\nabla_{\beta}(v+u_0)^B\nabla_{\gamma}(v+u_0)^A}{1+|\nabla(v+u_0)|^2},
		\end{split}\\
		w(\cdot,0)=0.	\label{D2}
	\end{numcases}
satisfying
\begin{equation}
\|v_1\|_{C^{2+\mu,1+\mu/2}(M\times[0,T])}\leq C(\mu,M,N)(\|v_1\|_{C^0(M\times[0,T])}+\|u_0\|_{C^{2+\nu}(M)}+\kappa_0).
\end{equation}
Since $v_1(\cdot,0)=0$, we have
\begin{equation}
\|v_1\|_{C^0(M\times[0,T])}\leq C(\mu,M,N)T(\|v_1\|_{C^0(M\times[0,T])}+\|u_0\|_{C^{2+\nu}(M)}+\kappa_0).
\end{equation}
By taking $T>0$ small enough, we get
\begin{equation}
\|v_1\|_{C^0(M\times[0,T])}\leq C(\mu,M,N)T(\|u_0\|_{C^{2+\nu}(M)}+\kappa_0).
\end{equation}
Then the interpolation inequality in \cite{lieberman1996second} implies that $v_1\in V^T$ for $T>0$ sufficiently small. For such $v_1$, we have $\psi(v_1+u_0)$ satisfying \eqref{C1+mu} and \eqref{Cmu}. Replacing $(v,\psi(v+u_0))$ in \eqref{D1}-\eqref{D2} by $(v_1,\psi(v_1+u_0))$, then we get $v_2\in V^T$.
Iterating this procedure, we get a solution $v_{k+1}$ of \eqref{D1}-\eqref{D2} with $(v,\psi(v+u_0))$ replacing by $(v_k,\psi(v_k+u_0))$, which satisfies
\begin{equation}
\|\psi(v_{k+1}+u_0)\|_{C^{\mu,\mu/2}(M)}\leq C(\mu, M, N, \kappa_0, \|u_0\|_{C^1(M)}).
\end{equation}
and
\begin{equation}
\|v_{k+1}\|_{C^{2+\mu,1+\mu/2}(M\times[0,T])}\leq C(\mu,M,N)(\|u_0\|_{C^{2+\nu}(M)}+\kappa_0).
\end{equation}
By passing to a subsequence, we know that $v_k$ converges to some $u$ in $C^{2,1}(M\times[0,T])$ and $\psi^{v_k+u_0}$ converges to some $\psi$ in $C^0(M\times[0,T])$. Then it is easy to see that $(u,\psi)$ is a solution of \eqref{system in Euclidean1}-\eqref{system in Euclidean2} with $u(\cdot,0)=u_0$ and $\psi(\cdot,0)=\psi_0$. 

{\bf Step 2:} $u(x,t)$ takes value in $N$ for any $(x,t)\in M\times[0,T]$.

Suppose $u\in C^{2,1}(M\times[0,T], \mathbb{R}^q)$ and  $\psi\in C^{\mu,\mu/2}(M\times[0,T], \Sigma M\otimes(\pi\circ u)^*TN)\cap L^\infty([0,T];C^{1+\mu}(M))$ satisfy \eqref{system in Euclidean1}-\eqref{system in Euclidean2}. In the following, we write $||\cdot||$ and $\langle\cdot,\cdot\rangle$ for the Euclidean norm and scalar product, respectively. Similarly, we write $||\cdot||_g$ and $\langle\cdot,\cdot\rangle_g$ for the norm and inner product of $(M,g)$, respectively. We define
\begin{equation}
\rho:\mathbb{R}^q\to\mathbb{R}^q
\end{equation}
by $\rho(z)=z-\pi(z)$ and 
 \begin{equation}
\varphi: M\times[0,T]\to\mathbb{R}
\end{equation}
by $\varphi(x,t)=||\rho(u(x,t))||^2=\sum\limits_{A=1}^q|\rho^A(u(x,t))|^2$. A direct computation yields
\begin{equation}
\begin{split}
(\frac{\partial}{\partial{t}}-\Delta)\varphi(x,t)&=-2\sum\limits_{A=1}^q||\nabla(\rho^A\circ u)(x,t)||_g^2\\
&\quad+2\langle\rho\circ u,-\pi^A_B(u)F^B_1(u)\rangle\\
&\quad+\frac{2}{\alpha(1+|\nabla{u}|^2)^{\alpha-1}}\langle\rho\circ u,\rho^A_B(u)F^B_2(u,\psi)\rangle\\
&\quad+\frac{4(\alpha-1)}{1+|\nabla{u}|^2}\langle\rho\circ u, \nabla^2_{\beta\gamma}u^C\nabla_{\beta}u^C\nabla_{\gamma}u^B\rho_B^A(u)\rangle,
\end{split}\end{equation}
where $F_1^A$ and $F_2^A$ are defined in \eqref{F1A} and \eqref{F2A}, respectively. 

Since $\rho\circ u\in T^\perp_{\pi\circ{u}}N$ and $(d\pi)_{u}: \mathbb{R}^q\to T_{\pi\circ{u}}N$, we have
\begin{equation}
\langle\rho\circ u,-\pi^A_B(u)F^B_1\rangle=\langle\rho\circ u,\rho^A_B(u)F^B_2\rangle=0.
\end{equation}
Together with
\begin{equation}\begin{split}
&\frac{4(\alpha-1)}{1+|\nabla{u}|^2}\langle\rho\circ u, \nabla^2_{\beta\gamma}u^C\nabla_{\beta}u^C\nabla_{\gamma}u^B\rho_B^A(u)\rangle\\
&\leq 4(\alpha-1)||u||_{C^2(M)}||\rho\circ{u}||||\nabla(\rho\circ{u})||\\
&\leq 2(\alpha-1)(||u||^2_{C^2(M)}\varphi+||\nabla(\rho\circ{u})||^2),
\end{split}\end{equation}
we get
$
(\frac{\partial}{\partial{t}}-\Delta)\varphi(x,t)\leq C\varphi,
$
where $C=C(\|u\|_{C^{2,1}(M\times[0,T])})$. Since $\varphi(x,t)\geq0$ and $\varphi(x,0)=0$ for any $(x,t)\in M\times[0,T]$, we conclude $\varphi=0$ on $M\times[0,T]$. We have shown that $u(x,t)\in N$ for all $(x,t)\in M\times[0,T]$. 

Finally, by using the $\epsilon$-regularity (see Lemma \ref{small energy regularity} below), we conclude that 
\begin{equation}
u\in C^{2+\mu,1+\mu/2}(M\times[0,T], N)
\end{equation} and 
\begin{equation}
\psi\in C^{\mu,\mu/2}(M\times[0,T], \Sigma M\otimes(\pi\circ u)^*TN)\cap L^\infty([0,T];C^{1+\mu}(M)).
\end{equation}
\end{proof}

Since the equations for $\alpha$-Dirac-harmonic maps are invariant under multiplying the spinor by elements of $\mathbb{H}$ with unit norm, by uniqueness we always mean  uniqueness up to multiplication of  the spinor by such elements. This kind of uniqueness for the Dirac-harmonic map flow was proved by the Banach fixed point theorem in \cite{wittmann2017short}. However, we cannot apply the fixed point theorem to the $\alpha$-Dirac-harmonic map flow. Therefore, it is interesting to consider the uniqueness of the $\alpha$-Dirac-harmonic map flow from closed surfaces. By considering the evolution inequality of $\|u_1-u_2\|_{C^0(M)}$, we can prove the following uniqueness which is weaker than that in \cite{wittmann2017short} because when the quaternions $h_a$ are different, we can no longer bound the $C^0$-norm of the difference of the maps.

\begin{thm}\label{uniqueness}
For any given $T>0$, let $(u_1,\psi_1)$ and $(u_2,\psi_2)$ be two solutions to \eqref{alpha map}-\label{alpha dirac} with the constraint \eqref{initial value} and $u_1,u_2\in C^{2+\mu,1+\mu/2}(M\times[0,T],N)$. Then there exists a time $T_1>0$, which depends on $R$ and the $C^{1+\mu,\frac{1+\mu}{2}}$ norms of $u_1$ and $u_2$, such that $u_1,u_2\in B^{T_1}_R$ and 
\begin{equation}
\psi_1(x,t)=h_1(t)\psi(u_a(x,t)), \  \psi_2(x,t)=h_2(t)\psi(u_a(x,t))
\end{equation}
for some $h_1(t),h_1(t)\in\mathbb{H}$ with unit length, where $\psi(u(x,t))$ is defined by \eqref{def of psi u}. Furthermore, if $h_1(t)=h_2(t)$ on $[0,T_2]$ for some $T_2\leq T_1$, then $(u_1,\psi_1)\equiv(u_2,\psi_2)$ on $M\times[0,T_2]$.
\end{thm}

\begin{proof}
By the assumptions, we have
\begin{equation}
\|u_a(\cdot,t)-u_0\|_{C^0(M)}\to0, \ \|\nabla{u_a}(\cdot,t)-\nabla{u_0}||_{C^0(M)}\to0
\end{equation}
for $a=1,2$. Therefore, for small enough $T_1$, $u_1,u_2\in B^{T_2}_R(\bar{u}_0)$. Since ${\rm dim}_{\mathbb{H}}(\slashed{D}^{u_a})=1$ for $a=1,2$, there exist $h_a(t)\in\mathbb{H}$ such that 
\begin{equation}\label{representation of psi}
\psi_a(x,t)=\psi(u_a(x,t))h_a(t)
\end{equation}
for all $t\in[0,\tilde{T}]$, where $\psi(u(x,t))$ is defined by \eqref{def of psi u}. Moreover, $h_a(t)$ is of unit length since $\|\psi_a\|_{L^2(M)}=\|\psi(u_a)\|_{L^2}=1$. 

Now, let us consider the uniqueness of the flow. First, by subtracting the equations of $u_1$ and $u_2$ and multiplying by $u_1-u_2$, we have
\begin{equation}\label{difference0}
\begin{split}
&\frac12\partial_t|u_1-u_2|^2-\frac12\Delta|u_1-u_2|^2+|\nabla(u_1-u_2)|^2\\
&=2(\alpha-1)\bigg\langle\frac{\nabla^2_{\beta\gamma}{u^i_1}\nabla_\beta{u^i_1}\nabla_\gamma{u_1}}{1+|\nabla{u_1}|^2}-\frac{\nabla^2_{\beta\gamma}{u^j_2}\nabla_\beta{u^j_2}\nabla_\gamma{u_2}}{1+|\nabla{u_2}|^2},u_1-u_2\bigg\rangle\\
&\quad-\langle II(\nabla{u_1},\nabla{u_1})-II(\nabla{u_2},\nabla{u_2}),u_1-u_2\rangle\\
&\quad-\langle R(\psi_1,\nabla{u_1}\cdot\psi_1)-R(\psi_2,\nabla{u_2}\cdot\psi_2),u_1-u_2\rangle.
\end{split}\end{equation}
In the sequel, we will estimate the terms on the right-hand side of the inequality \eqref{difference0}.
\begin{equation}
\begin{split}
&2(\alpha-1)\bigg\langle\frac{\nabla^2_{\beta\gamma}{u^i_1}\nabla_\beta{u^i_1}\nabla_\gamma{u_1}}{1+|\nabla{u_1}|^2}-\frac{\nabla^2_{\beta\gamma}{u^j_2}\nabla_\beta{u^j_2}\nabla_\gamma{u_2}}{1+|\nabla{u_2}|^2},u_1-u_2\bigg\rangle\\
&=2(\alpha-1)\bigg\langle\frac{\nabla^2_{\beta\gamma}(u^i_1-u^i_2)\nabla_\beta{u^i_1}\nabla_\gamma{u_1}}{1+|\nabla{u_1}|^2},u_1-u_2\bigg\rangle\\
&\quad+2(\alpha-1)\bigg\langle\nabla^2_{\beta\gamma}{u^i_2}\nabla_\beta{u^i_1}\nabla_\gamma{u_1}(\frac{1}{1+|\nabla{u_1}|^2}-\frac{1}{1+|\nabla{u_2}|^2}),u_1-u_2\bigg\rangle\\
&\quad+2(\alpha-1)\bigg\langle\frac{\nabla^2_{\beta\gamma}{u^i_2}\nabla_\gamma{u_1}}{1+|\nabla{u_2}|^2}(\nabla_\beta{u^i_1}-\nabla_\beta{u^i_2}),u_1-u_2\bigg\rangle\\
&\quad+2(\alpha-1)\bigg\langle\frac{\nabla^2_{\beta\gamma}{u^i_2}\nabla_\beta{u^i_2}}{1+|\nabla{u_2}|^2}(\nabla_\gamma{u_1}-\nabla_\gamma{u_2}),u_1-u_2\bigg\rangle\\
&\leq2(\alpha-1)\bigg\langle\frac{\nabla^2_{\beta\gamma}(u^i_1-u^i_2)\nabla_\beta{u^i_1}\nabla_\gamma{u_1}}{1+|\nabla{u_1}|^2},u_1-u_2\bigg\rangle\\
&\quad+C(\alpha-1)|\nabla(u_1-u_2)||u_1-u_2|,
\end{split}\end{equation}
where we used $u_1,u_2\in C^{2+\mu,1+\mu/2}(M\times[0,T],N)$. Similar, by the triangle inequality, we get
\begin{equation}
\begin{split}
&|\langle II(\nabla{u_1},\nabla{u_1})-II(\nabla{u_2},\nabla{u_2}),u_1-u_2\rangle|\\
&\leq C|u_1-u_2|^2+C|\nabla(u_1-u_2)||u_1-u_2|
\end{split}
\end{equation}
and
\begin{equation}
\begin{split}
&|\langle R(\psi_1,\nabla{u_1}\cdot\psi_1)-R(\psi_2,\nabla{u_2}\cdot\psi_2),u_1-u_2\rangle|\\
&\leq C|u_1-u_2|^2+C|\nabla(u_1-u_2)||u_1-u_2|+ C|\psi_1-\psi_2||u_1-u_2|.
\end{split}\end{equation}
Based on these estimates, \eqref{difference0} becomes 
\begin{equation}\label{difference1}
\begin{split}
&\frac12\partial_t|u_1-u_2|^2-\frac12\Delta|u_1-u_2|^2\\
&\leq2(\alpha-1)\bigg\langle\frac{\nabla^2_{\beta\gamma}(u^i_1-u^i_2)\nabla_\beta{u^i_1}\nabla_\gamma{u_1}}{1+|\nabla{u_1}|^2},u_1-u_2\bigg\rangle-|\nabla(u_1-u_2)|^2\\
&\quad+C|u_1-u_2|^2+C|\nabla(u_1-u_2)||u_1-u_2|+ C|\psi_1-\psi_2||u_1-u_2|.
\end{split}\end{equation}

Next, we want to bound those terms in the right-hand side of \eqref{difference1} by $|u_1-u_2|^2$ and $|\nabla{u_1}-\nabla{u_2}|^2$. Since $u_1,u_2\in B^{T_2}_R(\bar{u}_0)$, there is a unique  geodesic between $u_1(x,t)$ and $u_2(x,t)$ for any $(x,t)\in M\times[0,T_2]$. Now, for any $(x,t)\in P:=\{x\in M\times[0,T_2]|u_1(x,t)\neq u_2(x,t)\}$, we define
\begin{equation}
u_s(x,t):=\exp_{u_1(x,t)}(sv(x,t))=\exp_{u_1(x)}(sV(x,t)/|V(x,t)|)
\end{equation}
where $s\in[0,|V(x,t)|]$, $V(x,t):=\exp^{-1}_{u_1(x,t)}u_2(x,t)$ and $|V(x,t)|$ denotes the norm of $V(x,t)$ in the tangent space $T_{u_1(x,t)}N$. Then we can estimate $\nabla^2(u_1-u_2)$ as follows:
\begin{equation}
\begin{split}
\nabla^2_{\beta\gamma}(u_2-u_1)(x,t)&=\nabla^2_{\beta\gamma}u_{|V(x,t)|}(x,t)-\nabla^2_{\beta\gamma}u_0(x,t)\\
&=\int_0^{|V(x,t)|}\frac{d}{ds}\nabla^2_{\beta\gamma}u_s(x,t)\\
&\leq\sup_{[0,|V(x,t)|]\times P}\bigg|\frac{d}{ds}\nabla^2u_s\bigg|d^N(u_1(x,t),u_2(x,t))\\
&\leq C|u_1(x,t)-u_2(x,t)|,
\end{split}
\end{equation}
where we used the Lemma 5.1 in the Appendix. Hence, we can rewrite \eqref{difference1} as 
\begin{equation}\label{difference2}
\begin{split}
&\frac12\partial_t|u_1-u_2|^2-\frac12\Delta|u_1-u_2|^2\\
&\leq2(\alpha-1)\bigg\langle\frac{\nabla^2_{\beta\gamma}(u^i_1-u^i_2)\nabla_\beta{u^i_1}\nabla_\gamma{u_1}}{1+|\nabla{u_1}|^2},u_1-u_2\bigg\rangle-|\nabla(u_1-u_2)|^2\\
&\quad+C|u_1-u_2|^2+C|\nabla(u_1-u_2)||u_1-u_2|+C|\psi_1-\psi_2||u_1-u_2|\\
&\leq C|u_1-u_2|^2+C|\psi_1-\psi_2||u_1-u_2|,
\end{split}\end{equation}
where we used  Young's inequality. It remains to bound $|\psi_1-\psi_2|$ by $|u_1-u_2|$. To that end, we use the Lemma \ref{lip} and \eqref{representation of psi} as follows:
\begin{equation}
\begin{split}
|\psi_1-\psi_2|&=|h_1\psi(u_1)-h_2\psi_2(u_2)|\\
&=|\psi(u_1)-\psi(u_2)|\\
&\leq \|u_1-u_2\|_{C^0(M)},
\end{split}
\end{equation}where we used $h_1=h_2$ in the second equality.

Last, it is  easy to see $(u_1\psi_1)\equiv(u_2,\psi_2)$ by considering the following evolution inequality 
\begin{equation}
\partial_t\|u_1-u_2\|_{C^0(M)}^2\leq C\|u_1-u_2\|_{C^0(M)}^2
\end{equation}
with $u_1(\cdot,0)=u_2(\cdot,0)$.

\end{proof}

\subsection{Regularity of the flow}

In this subsection, we will give some estimates on the regularity of the flow. Let us start with the following estimate of the energy of the map part.

\begin{lem}\label{map energy}
Suppose $(u,\psi)$ is a solution of \eqref{alpha map}-\eqref{alpha dirac} with the initial values \eqref{initial value}. Then there holds
\begin{equation}
E^\alpha(u(t))+2\alpha\int_0^t\int_M(1+|\nabla{u}|^2)^{\alpha-1}|\partial_{t}u|^2=E^\alpha(u_0),
\end{equation}
where $E^\alpha(u):=\frac12\int_M(1+|\nabla{u}|^2)^{\alpha}$. Moreover, $E^\alpha(u(t))$ is absolutely continuous on $[0,T]$ and non-increasing.
\end{lem}
\begin{proof}
Note that \eqref{alpha map} can be written as:
\begin{equation}
\begin{split}
(1+|\nabla{u}|^2)^{\alpha-1}\partial_{t}u&={\rm div}((1+|\nabla_{u}|^2)^{\alpha-1}\nabla{u})-(1+|\nabla_g{u}|^2)^{\alpha-1}A(du,du)\\
&\quad-\frac{1}{\alpha}Re(P(\mathcal{A}(du(e_\beta),e_{\beta}\cdot\psi);\psi)).
\end{split}\end{equation}
Multiplying the inequality above by $\partial_t{u}$ and using  
\begin{equation}\begin{split}
0&=\int_0^t\int_M\langle\psi,\frac{d}{dt}\slashed{D}\psi\rangle\\
&=\int_0^t\int_M\langle\psi,\slashed{D}(\partial_t\psi)+e_{\gamma}\cdot\psi^i\otimes R^m_{ijk}\partial_t{u^j}du^k(e_{\gamma}))\partial_{y^m}\rangle\\
&=\int_0^t\int_MR_{mijk}\langle\psi^m,\nabla{u^k}\cdot\psi^i\rangle\partial_tu^j\\
&=\int_0^t\int_M[\langle S(\partial_{y^m},\partial_{y^j}),S(\partial_{y^i},\partial_{y^k})\rangle_{\mathbb{R}^q}-\langle S(\partial_{y^m},\partial_{y^k}),S(\partial_{y^i},\partial_{y^j})\rangle_{\mathbb{R}^q}]\\
&\quad\langle\psi^m,\nabla{u^k}\cdot\psi^i\rangle\partial_t{u^j}\\
&=2\int_0^t\int_M\langle S(\partial_{y^m},\partial_{y^j}),S(\partial_{y^i},\partial_{y^k})\rangle_{\mathbb{R}^q}Re(\langle\psi^m,\nabla{u^k}\cdot\psi^i\rangle)\partial_t{u^j}\\
&=2\int_0^t\int_M\langle Re(P(\mathcal{A}(du(e_\beta),e_{\beta}\cdot\psi);\psi)),\partial_t{u^j}\rangle,
\end{split}\end{equation}
we get
\begin{equation}
\begin{split}
\int_0^t\int_M(1+|\nabla{u}|^2)^{\alpha-1}|\partial_{t}u|^2&=\int_0^t\int_M\langle{\rm div}((1+|\nabla{u}|^2)^{\alpha-1}\nabla{u}),\partial_t{u}\rangle\\
&=-\int_0^t\int_M\langle(1+|\nabla_{g}u|^2)^{\alpha-1}\nabla{u},\partial_t\nabla{u}\rangle\\
&=-\frac{1}{2\alpha}\int_0^t\frac{d}{dt}\int_M(1+|\nabla{u}|^2)^\alpha,
\end{split}\end{equation}
which directly gives us the lemma.
\end{proof}

Consequently, we can also control the spinor part along the heat flow of the $\alpha$-Dirac-harmonic map.
\begin{lem}
Suppose $(u,\psi)$ is a solution of \eqref{alpha map}-\eqref{alpha dirac} with the initial values \eqref{initial value}. Then for any $p\in(1,2)$, there holds
\begin{equation}
||\psi(\cdot,t)||_{W^{1,p}(M)}\leq C, \ \forall t\in[0,T],
\end{equation}
where $C=C(p,M,N,E^\alpha(u_0))$.
\end{lem}
\begin{proof}
The lemma directly follows from Lemma \ref{map energy} and the following lemma:
\begin{lem}\label{key estimate}
   Let $M$ be a closed spin Riemann surface, N be a compact Riemann manifold. Let $u\in W^{1,2\alpha}(M,N)$ for some $\alpha>1$ and $\psi\in W^{1,p}(M, \Sigma M\otimes u^*TN)$ for $1<p<2$, then there exists a positive constant $C=C(p,M,N,\|\nabla u\|_{L^{2\alpha}})$ such that
 \begin{equation}
 \|\psi\|_{W^{1,p}(M)}\leq C(\|\slashed{D}\psi\|_{L^{p}(M)}+\|\psi\|_{L^p(M)}).
 \end{equation}
 \end{lem}
 This lemma follows from applying Lemma \ref{spinor norm with boundary} to $\eta\psi$, where $\eta$ is a cut-off function.
\end{proof}

To get the convergence of the flow, we also need the following $\epsilon$-regualrity.
\begin{lem}\label{small energy regularity}
Suppose $(u,\psi)$ is a solution of \eqref{alpha map}-\eqref{alpha dirac} with the initial values \eqref{initial value}. Given $\omega_0=(x_0,t_0)\in M\times(0,T]$, denote
\begin{equation}
P_R(\omega_0):=B_R(x_0)\times[t_0-R^2,t_0].
\end{equation}
Then there exist three constants $\epsilon_2=\epsilon_2(M,N)>0$, $\epsilon_3=\epsilon_3(M,N,u_0)>0$ and $C=C(\mu,R,M,N,E^\alpha(u_0))>0$ such that if
\begin{equation}
1<\alpha<1+\epsilon_2, \ \text{and} \sup_{[t_0-4R^2,t_0]}E(u(t);B_{2R}(\omega_0))\leq\epsilon_3,
\end{equation}
then 
\begin{equation}
\sqrt{R}||\psi||_{L^\infty(P_R(\omega_0))}+R||\nabla{u}||_{L^\infty(P_R(\omega_0))}\leq C
\end{equation}
and for any $0<\beta<1$, 
\begin{equation}
\sup_{[t_0-\frac{R^2}{4},t_0]}||\psi(t)||_{C^{1+\mu}(B_{R/2}(x_0))}+||\nabla{u}||_{C^{\beta,\beta/2}(P_{R/2}(\omega_0))}\leq C(\beta).
\end{equation}
Moreover, if 
\begin{equation}
\sup_{M}\sup_{[t_0-4R^2,t_0]}E(u(t);B_{2R}(\omega_0))\leq\epsilon_3,
\end{equation}
then 
\begin{equation}
||u||_{C^{2+\mu,1+\mu/2}(M\times[t_0-\frac{R^2}{8},t_0])}+||\psi||_{C^{\mu,\mu/2}(M\times[t_0-\frac{R^2}{8},t_0])}+\sup_{[t_0-\frac{R^2}{8},t_0]}||\psi(t)||_{C^{1+\mu}(M)}\leq C.
\end{equation}
\end{lem}
Since $M$ is closed, $x_0$ has to be an interior point of $M$. Therefore, our Lemma is just a special case of  the Lemma 3.4 in \cite{jost2018geometric}. So we omit the proof here.

\section{Existence of $\alpha$-Dirac-harmonic maps}
In this section, we will prove Theorem \ref{existence of DH maps} by the following theorem on the existence of $\alpha$-Dirac-harmonic maps for $\alpha>1$.

\begin{thm}\label{Existence of alpha DH map}
Let $M$ be a closed spin surface and $(N,h)$ a real analytic closed manifold. Suppose there exists a map $u_0\in C^{2+\mu}(M,N)$ for some $\mu\in(0,1)$ such that ${\rm dim}_{\mathbb{H}}{\rm ker}\slashed{D}^{u_0}=1$. Then for any $\alpha\in(1,1+\epsilon_1)$, there exists a nontrivial smooth $\alpha$-Dirac-harmonic map $(u_\alpha,\psi_\alpha)$ such that the map part $u_\alpha$ stays in the same homotopy class as $u_0$ {and $\|\psi_\alpha\|_{L^2}=1$. }
\end{thm}

\begin{proof}[Proof of Theorem \ref{Existence of alpha DH map}]
Let us denote the energy minimizer by 
\begin{equation}
m^\alpha_0:=\inf\{E^\alpha(u)|u\in W^{1,2\alpha}(M,N)\cap[u_0]\},
\end{equation}
where $[u_0]$ denotes the homotopy class of $u_0$. If $u_0$ is a minimizing $\alpha$-harmonic map, it follows from Lemma \ref{map energy} that $(u_0,\psi_0)$ is an $\alpha$-Dirac-harmonic map for any $\psi_0\in{\rm ker}\slashed{D}^{u_0}$. If $E^\alpha(u_0)>m^\alpha_0$, then Theorem \ref{short-time alpha flow} gives us a solution
\begin{equation}
u\in C^{2+\mu,1+\mu/2}(M\times[0,T),N)
\end{equation}
and
\begin{equation}
\psi\in C^{\mu,\mu/2}(M\times[0,T), \Sigma M\otimes u^*TN)\cap{\cap_{0<s<T}}L^\infty([0,s];C^{1+\mu}(M)).
\end{equation}
to the problem \eqref{alpha map}-\eqref{alpha dirac} with the initial values \eqref{initial value}.

 By Lemma \ref{map energy}, we know
\begin{equation}
\int_M(1+|\nabla{u}|^2)^\alpha\leq E^\alpha(u_0).
\end{equation}
Then it is easy to see that, for any $0<\epsilon<\epsilon_3$, there exists a positive constant $r_0=r_0(\epsilon, \alpha, E^\alpha(u_0))$ such that for all $(x,t)\in M\times[0,T)$, there holds
\begin{equation}
\int_{B_{r_0}(x)}|\nabla{u}|^2\leq CE^\alpha(u_0)^{1/\alpha}r_0^{1-\frac{1}{\alpha}}\leq \epsilon.
\end{equation}
Therefore, by Theorem \ref{short-time alpha flow} and Lemma \ref{small energy regularity}, we know that the singular time can be characterized as
\begin{equation}
Z=\{T\in\mathbb{R}|\lim\limits_{t_i\nearrow T}{\rm dim}_{\mathbb{H}}{\rm ker}\slashed{D}^{u_{t_i}}>1\}
\end{equation}
and there exists a sequence $\{t_i\}\nearrow T$ such that 
\begin{equation}
(u(\cdot,t_i),\psi(\cdot,t_i))\to (u(\cdot,T),\psi(\cdot,T)) \ \text{in} \ C^{2+\mu}(M)\times C^{1+\mu/2}(M)
\end{equation}
{and
\begin{equation}
\|\psi(\cdot,T)\|_{L^2}=1.
\end{equation}
}

If $Z=\emptyset$, then, by Theorem \ref{short-time alpha flow}, we can extend the solution $(u,\psi)$ beyond the time $T$ by using $(u(\cdot,T),\psi(\cdot,T))$ as new initial values. Thus, we have the global existence of the flow. For the limit behavior as $t\to\infty$,  Lemma \ref{map energy} implies that there exists a sequence $\{t_i\}\to\infty$ such that 
\begin{equation}\label{time derivative}
\int_M|\partial_tu|^2(\cdot,t_i)\to 0.
\end{equation}
Together with  Lemma \ref{small energy regularity}, there is a subsequence, still denoted by $\{t_i\}$, and an $\alpha$-Dirac-harmonic map $(u_\alpha, \psi_\alpha)\in C^\infty(M,N)\times C^\infty(M,\Sigma M\otimes(u_\alpha)^*TN)$ such that $(u(\cdot,t_i),\psi(\cdot,t_i))$ converges to $(u_\alpha, \psi_\alpha)$ in $C^2(M)\times C^1(M)$ {and $\|\psi_\alpha\|_{L^2}=1$.}

If $Z\neq\emptyset$ and $T\in{Z}$, let us assume that $E^\alpha(u(\cdot,T))>m^\alpha_0$ and $(u(\cdot,T),\psi(\cdot,T))$ is not already an $\alpha$-Dirac-harmonic map.  We extend the flow as follows: By Lemma \ref{density}, there is a map $u_1\in C^{2+\mu}(M,N)$ such that 
 \begin{equation}\label{lower energy}
 m^\alpha_0<E^\alpha(u_1)<E^\alpha(u(\cdot,T))
 \end{equation}
 and 
 \begin{equation}\label{minimal kernel}
 {\rm dim}_{\mathbb{H}}{\rm ker}\slashed{D}^{u_1}=1.
 \end{equation}
Thus, {picking any $\psi_1\in{\rm ker}\slashed{D}^{u_1}$ with $\|\psi_1\|_{L^2}=1$}, we can restart the flow from the new initial values $(u_1,\psi_1)$. If there is no singular time along the flow started from $(u_1,\psi_1)$, then we get an $\alpha$-Dirac-harmonic map as in the case of $Z=\emptyset$. Otherwise, we use again the procedure above to choose $(u_2,\psi_2)$ as  initial values and restart the flow.  {This procedure will stop in  finitely or infinitely many steps.}

 If infinitely many steps are required, then there exist infinitely many  flow pieces  $\{u_i(x,t)\}_{i=1,\dots,\infty}$ and $\{T_i\}_{i=1,\dots,\infty}$ such that
\begin{equation}
E^\alpha(u_i(t))+2\alpha\int_0^t\int_M(1+|\nabla{u}|^2)^{\alpha-1}|\partial_{t}u|^2=E^\alpha(u_i), \ \forall t\in(0,T_i),
\end{equation}
where $u_i(\cdot,0)=u_i\in C^{2+\mu}(M,N)$.  If the $T_i$ are bounded away from zero, then there is $\{t_i\}$ such that \eqref{time derivative} hold for $t_i\in(0,T_i)$. Therefore, we have an $\alpha$-Dirac-harmonic map as before. If $T_i\to 0$, then we look at the limit of $E^\alpha(u_i)$. If the limit is strictly bigger than $m^\alpha_0$, we again choose another map satisfying \eqref{lower energy} and \eqref{minimal kernel} as a new starting point. If the limit is exactly $m^\alpha_0$, then we choose $\{t_i\}$ such that $t_i\in(0,T_i)$ for each $i$. By Lemma \ref{small energy regularity},  $u_i(t_i)$ converges in $C^2(M)\times C^1(M)$ to a minimizing $\alpha$-harmonic map $u_\alpha$. If $\slashed{D}^{u_\alpha}$ has minimal kernel, then for any $\psi\in{\rm ker}{\slashed{D}^{u_\alpha}}$, $(u_\alpha,\psi)$ is an $\alpha$-Dirac-harmonic map as we showed in the beginning of the proof. If $\slashed{D}^{u_\alpha}$ has non-minimal kernel, we use the decomposition of the twisted spinor bundle through the $\mathbb{Z}_2$-grading $G\otimes id$ (see \cite{ammann2013dirac}). {More precisely, for any smooth variation $(u_s)_{s\in(-\epsilon,\epsilon)}$ of $u_0$, we split the bundle $\Sigma M\otimes u_s^*TN$ into $\Sigma M\otimes u_s^*TN=\Sigma^+M\otimes u_s^*TN\oplus\Sigma^-M\otimes u_s^*TN$, which is orthogonal in the complex sense and parallel. Consequently, for any $\psi_0\in{\rm ker}{\slashed{D}^{u_0}}$, we have
\begin{equation} 
(\slashed{D}^{u_0}\psi_0^+,\psi_0^+)_{L^2}=(\slashed{D}^{u_0}\psi_0^-,\psi_0^-)_{L^2}=0
\end{equation}
for $\psi_0=\psi_0^++\psi_0^-$, where $\psi_0^{\pm}=\psi_{\pm}\otimes u_0^*TN$ and $\psi_{\pm}\in\Sigma^{\pm}$. Therefore, $\psi_s^{\pm}:=\psi_{\pm}\otimes u_s^*TN$ are smooth variations of $\psi_0^\pm$, respectively, such that 
\begin{equation}
\frac{d}{dt}\bigg|_{t=0}(\slashed{D}^{u_s}\psi_s^\pm,\psi_s^\pm)_{L^2}=0.
\end{equation} 
By taking $u_0=u_\alpha$ and $\psi_0=\psi_\alpha\in{\rm ker}{\slashed{D}^{u_\alpha}}$, the first variation formula of $L^\alpha$ implies that $(u_\alpha,\psi_\alpha^\pm)$ are $\alpha$-Dirac-harmonic maps (see Corollary 5.2 in \cite{ammann2013dirac}). In particular, we can choose $\psi_\alpha$ such that $\|\psi_\alpha^+\|_{L^2}=1$ or $\|\psi_\alpha^-\|=1$.}

If it stops in finitely many steps, there exists a sequence $\{t_i\}$ and some $0<T_k\leq+\infty$ such that
\begin{equation}
\lim\limits_{t_i\nearrow T}(u(\cdot,t_i),\psi(\cdot,t_i))\to (u_\alpha,\psi_\alpha) \ \text{in} \ C^2(M)\times C^1(M),
\end{equation}
where $(u_\alpha,\psi_\alpha)$ either is an $\alpha$-Dirac-harmonic map or satisfies $E^\alpha(u_\alpha)=m^\alpha_0$. And in the latter case, $u_\alpha$ is a minimizing $\alpha$-harmonic map. Then we can again get a nontrivial $\alpha$-Dirac-harmonic map as above.
\end{proof}

By Theorem \ref{Existence of alpha DH map}, for any $\alpha>1$ sufficiently close to $1$, there exists an $\alpha$-Dirac-harmonic map $(u_\alpha,\psi_\alpha)$ with the properties
\begin{equation}\label{map energy bound}
E^\alpha(u_\alpha)\leq E^\alpha(u_0), \  \ \|\psi_\alpha\|_{L^2}=1 
\end{equation}
and 
\begin{equation}\label{Lp of spinor}
||\psi_\alpha||_{W^{1,p}(M)}\leq C(p,M,N,E^\alpha(u_0))
\end{equation}
for any $1<p<2$. Then it is natural to consider the limit behavior when $\alpha$ decreases to $1$. Since the blow-up analysis was already well studied in \cite{jost2018geometric}, we can directly prove Theorem \ref{existence of DH maps}.

\begin{proof}[Proof of Theorem \ref{existence of DH maps}]
By Theorem \ref{Existence of alpha DH map}, we have a sequence of smooth $\alpha$-Dirac-harmonic maps $(u_{\alpha_k},\psi_{\alpha_k})$ with \eqref{map energy bound} and \eqref{Lp of spinor}, where $\alpha_k\searrow1$ as $k\to\infty$. Then, by Theorem 2.1 in \cite{jost2018geometric}, there is a constant $\epsilon_0>0$ and a Dirac-harmonic map 
$$
(\Phi,\Psi)\in C^\infty(M,N)\times C^\infty(M,\Sigma M\otimes\Phi^*TN)
$$
such that
\begin{equation}
(u_{\alpha_k},\psi_{\alpha_k})\to(\Phi,\Psi) \ \text{in} \  C^2_{loc}(M\setminus{\mathcal S})\times C^1_{loc}(M\setminus{\mathcal S}),
\end{equation}
where 
\begin{equation}
\mathcal{S}:=\{x\in M|\liminf_{\alpha_k\to1}E(u_{\alpha_k};B_{r}(x))\geq\frac{\epsilon_0}{2}, \forall r>0 \}
\end{equation}
is a finite set.  

Now, taking $x_0\in\mathcal{S}$, there exists a sequence $x_{\alpha_k}\to x_0$, $\lambda_{\alpha_k}\to0$ and a nontrivial Dirac-harmonic map $(\phi,\xi): \mathbb{R}^2\to N$ such that 
\begin{equation}
(u_{\alpha_k}(x_{\alpha_k}+\lambda_{\alpha_k}x),\lambda_{\alpha_k}^{{\alpha_k}-1}\sqrt{\lambda_{\alpha_k}}\psi_{\alpha_k}(x_{\alpha_k}+\lambda_{\alpha_k}x))\to(\phi,\xi) \ \text{in} \ C^2_{loc}(\mathbb{R}^2),
\end{equation}
as $\alpha\to1$. Choose any $p^*>4$, by taking $p=\frac{2p^*}{2+p^*}$ in \eqref{Lp of spinor}, we get
\begin{equation}
||\psi_{\alpha_k}||_{L^{p^*}(M)}\leq C(p^*,M,N,E^{\alpha_k}(u_0))
\end{equation}
and 
\begin{equation}
||\xi||_{L^4(D_R(0))}=\lim\limits_{{\alpha_k}\to1}\lambda^{{\alpha_k}-1}_{\alpha_k}||\psi_{\alpha_k}||_{L^4(D_{\lambda_{\alpha_k}{R}}(x_{\alpha_k}))}\leq\lim\limits_{{\alpha_k}\to1}C||\psi_{\alpha_k}||_{L^{p^*}(M)}(\lambda_{\alpha_k}{R})^{2(\frac14-\frac{1}{p^*})}=0.
\end{equation}
Thus, $\xi=0$ and $\phi$ can be extended to a nontrivial smooth harmonic sphere. Since $||\psi_\alpha||_{L^2}=1$, the Sobolev embedding implies that $||\Psi||_{L^2(M)}=\lim\limits_{{\alpha_k}\to1}||\psi_\alpha||_{L^2(M)}=1$. Therefore, $(\Phi,\Psi)$ is nontrivial. Furthermore, if $(N,h)$ does not admit any nontrivial harmonic sphere, then 
\begin{equation}
(u_{\alpha_k},\psi_{\alpha_k})\to(\Phi,\Psi) \ \text{in} \  C^2(M)\times C^1(M).
\end{equation}
Therefore, $\Phi$ is in the same homotopy class as $u_0$.
\end{proof}

\section{Appendix}
In Section 3, we used some convenient properties of the elements in $B^T_R(\bar{u}_0)$. Those properties were already discussed in \cite{wittmann2017short}. However, the function space used there is $B^T_R({v}_0)$, where $v_0(x,t)=\int_Mp(x,y,t)u_0(y)dV(y)$, because the solution there is the unique fixed point of the following integral representation over $B^T_R({v}_0)$
\begin{equation}
Lu(x,t):=v_0(x,t)+\int_0^t\int_Mp(x,y,t-\tau)(F_1(u_\tau)+F_2(u_\tau,\psi(u_\tau)))dV(y)d\tau
\end{equation}
where $p$ is the heat kernel of $M$, $F_1$ and $F_2$ are defined as in \eqref{F1A} and \eqref{F2A}, respectively. Our proof for the short-time existence is different from there, and the space $B^T_R(\bar{u}_0)$ is more natural and convenient in our situation. Therefore, we cannot directly use the statement in \cite{wittmann2017short}. Although the space is changed, the proofs of those nice properties are parallel. In fact, one can see from the following that to make the elements in $B^T_R(\bar{u}_0)$ satisfy nice properties \eqref{small nbhd} and \eqref{shortest geodesic}, it is sufficient to choose $R$ small, namely, $T$ is independent of $R$. This is the biggest advantage. In the following, we will give the precise statement of the  properties we need in Section 3 and  proofs for the most important lemmas. 

For every $T>0$, we consider the space $B^T_R(\bar{u}_0):=\{u\in X_T|\|u-\bar{u}_0\|_{X_T}\leq R\}\cap\{u|_{t=0}=u_0\}$ where $\bar{u}_0(x,t)=u_0(x)$ for any $t\in[0,T]$. To get the necessary estimate for the solution of the constraint equation, we will use the parallel transport along the unique shortest geodesic between $u_0(x)$ and $\pi\circ{u_t}(x)$ in N. To do this, we need the following lemma which tells us that the distances in $N$ can be locally controlled by the distances in $\mathbb{R}^q$.

\begin{lem}\cite{wittmann2017short}\label{distance control}
Let $N\subset\mathbb{R}^q$ be a closed embedded submanifold of $\mathbb{R}^q$ with the induced Riemannian metric. Denote by $A$ its Weingarten map. Choose $C>0$ such that $||A||\leq C$, where
\begin{equation} 
||A||:=\sup\{||A_vX||| \ v\in T_p^\perp{N}, \ X\in T_pN, \ ||v||=1, \ ||X||=1, \ p\in N\}.
\end{equation}
Then there exists $0<\delta_0<\frac1C$ such that for all $0<\delta\leq\delta_0$ and for all $p,q\in N$ with $||p-q||_2<\delta$, it holds that 
\begin{equation}
d^N(p,q)\leq\frac{1}{1-\delta{C}}||p-q||_2,
\end{equation}
where we denote the Euclidean norm by $||\cdot||_2$ in this section.
\end{lem}

In the following, we will choose $\delta$ and $R$ to ensure the existence of the unique shortest geodesics between the projections of any two elements in $B^T_R(\bar{u}_0)$. By the definition of $B^T_R(\bar{u}_0)$, we have 
\begin{equation}
||u(x,t)-\bar{u}_0(x,t)||_2=||u(x,t)-u_0(x)||_2\leq R
\end{equation}
for all $(x,t)\in M\times[0,T]$. Then taking any $R\leq\delta$, we get
\begin{equation}
d(u(x,t),N)\leq||u(x,t)-u_0(x)||_2\leq\delta
\end{equation}
for all $(x,t)\in M\times[0,T]$. Therefore, $u(x,t)\in N_\delta$. In particular, $\pi\circ{u}$ is $N$-valued, and 
\begin{equation}\label{proj u}
||(\pi\circ{u})(x,t)-u_0(x)||_2\leq||(\pi\circ{u})(x,t)-u(x,t)||_2+||u(x,t)-u_0(x)||_2\leq2\delta.
\end{equation}
Now, we choose $\epsilon>0$ with $2\epsilon<{\rm inj}(N)$ and $\delta$ such that 
\begin{equation}\label{delta}
\delta<\min\{\frac14\delta_0,\frac14\epsilon(1-\delta_0C)\}，
\end{equation}
where $\delta_0,C>0$ are as in Lemma \ref{distance control}. From \eqref{proj u}, we know that for all $u,v\in B^T_R(\bar{u}_0)$, it holds that
\begin{equation}
||(\pi\circ{u})(x,t)-(\pi\circ{v})(x,s)||_2\leq4\delta<\delta_0.
\end{equation}
Then Lemma \ref{distance control} and \eqref{delta} imply that 
\begin{equation}\label{control distance on N}
\begin{split}
d^N((\pi\circ{u})(x,t),(\pi\circ{v})(x,s))&\leq\frac{1}{1-\delta_0C}||(\pi\circ{u})(x,t)-(\pi\circ{v})(x,s)||_2\\
&\leq\frac{1}{1-\delta_0C}4\delta<\epsilon<\frac12{\rm inj}(N).
\end{split}\end{equation}
To summarize, under the choice of constants as follows:
\begin{equation}\label{choice of constants}
		\begin{cases}
		\epsilon>0, & \text{s.t.} \ 2\epsilon<{\rm inj}(N), \\
		\delta>0, & \text{s.t.} \ \delta<\min\{\frac14\delta_0,\frac14\epsilon(1-\delta_0C)\}, \\
		R\leq\delta,
				\end{cases}
	\end{equation}
we have shown that 
\begin{equation}\label{small nbhd}
u(x,t)\in N_\delta
\end{equation}
and 
\begin{equation}\label{shortest geodesic}
d^N((\pi\circ u)(x,t),(\pi\circ v)(x,s))<\epsilon<\frac12{\rm inj}(N)
\end{equation}
for all $u,v\in B_R^T(\bar{u}_0)$, $x\in M$ and $t,s\in[0,T]$. 

Using the  properties \eqref{small nbhd} and \eqref{shortest geodesic}, we can parallelly prove two important estimates as in \cite{wittmann2017short}. One is for the Dirac operators along maps.

\begin{lem}\label{Dirac along maps u}
Choose $\epsilon$, $\delta$ and $R$ as in \eqref{choice of constants}. If $\epsilon>0$ is small enough, then there exists $C=C(R)>0$ such that 
\begin{equation}
||((P^{v_s,u_t})^{-1}\slashed{D}^{\pi\circ{u_t}}P^{v_s,u_t}-\slashed{D}^{\pi\circ{v_s}})\psi(x)||\leq C||u_t-v_s||_{C^0(M,\mathbb{R}^q)}||\psi(x)||
\end{equation}
for any $u,v\in B^T_R(\bar{u}_0)$, $\psi\in\Gamma_{C^1}(\Sigma M\otimes(\pi\circ{v_s})^*TN)$, $x\in M$ and $t,s\in[0,T]$.
\end{lem}

\begin{proof}
We write $f_0:=\pi\circ{v_s}$, $f_1:=\pi\circ{u_t}$ and define the $C^1$ map $F: M\times[0,1]\to N$ by 
\begin{equation}
F(x,t):=\exp_{f_0(x)}(t\exp^{-1}_{f_0(x)}f_1(x))
\end{equation}
where $\exp$ denotes the exponential map of the Riemannian manifold $N$. Note that $F(\cdot,0)=f_0$, $F(\cdot,1)=f_1$ and $t\mapsto F(x,t)$ is the unique shortest geodesic from $f_0(x)$ to $f_1(x)$. We denote by 
\begin{equation}
\mathcal{P}_{t_1,t_2}=\mathcal{P}_{t_1,t_2}(x): T_{F(x,t_1)}N\to T_{F(x,t_2)}N
\end{equation}
the parallel transport in $F^*TN$ with respect to $\nabla^{F^*TN}$ (pullback of the Levi-Civita connection on $N$) along the curve $\gamma_x(t):=(x,t)$ from $\gamma_x(t_1)$ to $\gamma_x(t_2)$, $x\in M$, $t_1,t_2\in[0,1]$. In particular, $\mathcal{P}_{0,1}=P^{v_s,u_t}$. Let $\psi\in\Gamma_{C^1}(\Sigma M\otimes(f_0)^*TN)$. We have
\begin{equation}\label{difference of Dirac operaor}
\begin{split}
&((\mathcal{P}_{0,1})^{-1}\slashed{D}^{f_1}\mathcal{P}_{0,1}-\slashed{D}^{f_0})\psi\\
&=(e_\alpha\cdot\psi^i)\otimes(((\mathcal{P}_{0,1})^{-1}\nabla_{e_\alpha}^{f_1^*TN}\mathcal{P}_{0,1}-\nabla_{e_\alpha}^{f_0^*TN})(b_i\circ{f_0}))
\end{split}
\end{equation}
where $\psi=\psi^i\otimes(b_i\circ{f_0})$, $\{b_i\}$ is an orthonormal frame of $TN$, $\psi^i$ are local $C^1$ sections of $\Sigma M$, and $\{e_\alpha\}$ is an orthonormal frame of $TM$.

We define local $C^1$ sections $\Theta_i$ of $F^*TN$ by 
\begin{equation}
\Theta_i(x,t):=\mathcal{P}_{0,t}(x)(b_i\circ{f_0})(x).
\end{equation}
For each $t\in[0,1]$ we define the functions $T_{ij}(\cdot,t):=T_{ij}^\alpha(\cdot,t)$ by
\begin{equation}\label{Tij}
(\mathcal{P}_{0,t})^{-1}((\nabla_{e_\alpha}^{F^*TN}\Theta_i)(x,t))=\sum_jT_{ij}^\alpha(x,t)(b_j\circ{f_0})(x).
\end{equation}
So far, we only know that the $T_{ij}$ are continuous. In the following, we will perform some formal calculations and justify them afterwards. By a straightforward computation, we have 
\begin{equation}\label{difference of T}
\begin{split}
&||((\mathcal{P}_{0,1})^{-1}\nabla_{e_\alpha}^{f_1^*TN}\mathcal{P}_{0,1}-\nabla_{e_\alpha}^{f_0^*TN})(b_i\circ{f_0})(x)||_h^2\\
&=||(\mathcal{P}_{0,1})^{-1}((\nabla_{e_\alpha}^{F^*TN}\Theta_i)(x,1))-(\mathcal{P}_{0,0})^{-1}((\nabla_{e_\alpha}^{F^*TN}\Theta_i)(x,0))||_h^2\\
&=||\sum_jT_{ij}(x,1)(b_j\circ{f_0})(x)-\sum_jT_{ij}(x,0)(b_j\circ{f_0})(x)||_h^2\\
&=\sum_j(T_{ij}(x,1)-T_{ij}(x,0))^2\\
&=\sum_j\left(\int_0^1\frac{d}{dt}\bigg|_{t=r}T_{ij}(x,t)dr\right)^2.
\end{split}\end{equation}
Therefore we want to control the first time-derivative of the $T_{ij}$. Equation \eqref{Tij} implies that these time-derivatives are related to the curvature of $F^*TN$. More precisely, for all $X\in\Gamma(TM)$ we have
\begin{equation}\label{derivative of T}
\begin{split}
&\frac{d}{dt}\bigg|_{t=r}\left((\mathcal{P}_{0,t})^{-1}\left((\nabla_{X}^{F^*TN}\Theta_i)(x,t)\right)\right)\\
&=\frac{d}{dt}\bigg|_{t=0}\left((\mathcal{P}_{0,t+r})^{-1}\left((\nabla_{X}^{F^*TN}\Theta_i)(x,t+r)\right)\right)\\
&=\frac{d}{dt}\bigg|_{t=0}\left((\mathcal{P}_{0,r})^{-1}(\mathcal{P}_{r,r+t})^{-1}\left((\nabla_{X}^{F^*TN}\Theta_i)(x,t+r)\right)\right)\\
&=(\mathcal{P}_{0,r})^{-1}\frac{d}{dt}\bigg|_{t=0}\left((\mathcal{P}_{r,r+t})^{-1}\left((\nabla_{X}^{F^*TN}\Theta_i)(x,t+r)\right)\right)\\
&=(\mathcal{P}_{0,r})^{-1}\left((\nabla_{\frac{\partial}{\partial t}}^{F^*TN}\nabla_{X}^{F^*TN}\Theta_i)(x,r)\right).
\end{split}\end{equation}
Now, let us justify the formal calculations \eqref{difference of T} and \eqref{derivative of T}. Combining the definition of $\Theta_i$ as parallel transport and a careful examination of the regularity of F we
deduce that $(\nabla_{\frac{\partial}{\partial t}}^{F^*TN}\nabla_{X}^{F^*TN}\Theta_i)(x,r)$ exists. Then \eqref{derivative of T} holds. Together with \eqref{Tij}, we know that the $T_{ij}$ are differentiable in $t$. Therefore \eqref{difference of T} also holds. We further get 
\begin{equation}
\begin{split}
\nabla_{\frac{\partial}{\partial t}}^{F^*TN}\nabla_{X}^{F^*TN}\Theta_i
&=R^{F^*TN}(\frac{\partial}{\partial t},X)\Theta_i+\nabla_{X}^{F^*TN}\nabla_{\frac{\partial}{\partial t}}^{F^*TN}\Theta_i-\nabla_{[\frac{\partial}{\partial t},X]}^{F^*TN}\Theta_i\\
&=R^{F^*TN}(\frac{\partial}{\partial t},X)\Theta_i=R^{TN}(dF(\frac{\partial}{\partial t}),dF(X))\Theta_i,
\end{split}\end{equation}
since $\nabla_{\frac{\partial}{\partial t}}^{F^*TN}\Theta_i=0$ by the definition of $\Theta_i$ and $[\frac{\partial}{\partial t},X]=0$.

This implies 
\begin{equation}
\begin{split}
\sum_j\left(\frac{d}{dt}\bigg|_{t=r}T_{ij}(x,t)\right)^2
&=||\frac{d}{dt}\bigg|_{t=r}\left((\mathcal{P}_{0,t})^{-1}((\nabla_{e_\alpha}^{F^*TN}\Theta_i)(x,t))\right)||^2_h\\
&=||\left(\nabla_{\frac{\partial}{\partial t}}^{F^*TN}\nabla_{e_\alpha}^{F^*TN}\Theta_i\right)(x,r)||_h^2\\
&=||R^{TN}(dF_{(x,r)}(\frac{\partial}{\partial t}),dF_{(x,r)}(e_\alpha))\Theta_i(x,r)||_h^2\\
&\leq C_1||dF_{(x,r)}({\partial_t})||_h^2||dF_{(x,r)}(e_\alpha))||_h^2,
\end{split}\end{equation}
where $C_1$ only depends on $N$.

In the following we estimate $||dF_{(x,r)}({\partial_t})||_h$ and $||dF_{(x,r)}(e_\alpha))||_h$. We have 
\begin{equation}
dF_{(x,r)}({\partial_t}|_{(x,r)})=\frac{\partial}{\partial t}\bigg|_{t=r}(\exp_{f_0(x)}(t\exp^{-1}_{f_0(x)}f_1(x)))=c'(r),
\end{equation}
where $c(t):=\exp_{f_0(x)}(t\exp^{-1}_{f_0(x)}f_1(x))$ is a geodesic in $N$. In particular, $c'$ is parallel along $c$ and thus $||c'(r)||_h=||c'(0)||_h=||\exp^{-1}_{f_0(x)}f_1(x)||_h$. Therefore, we get 
\begin{equation}
||dF_{(x,r)}({\partial_t})||_h=||\exp^{-1}_{f_0(x)}f_1(x)||_h\leq d^N(f_0(x),f_1(x))\leq C_2||u_t-v_s||_{C^0(M,\mathbb{R}^q)},
\end{equation}
where we have used Lemma \ref{distance control} and the Lipschitz continuity of $\pi$. Moreover, there exists $C_3(R)>0$ such that $||dF_{(x,r)}(e_\alpha))||_h\leq C_3(R)$ for all $(x,r)\in M\times[0,1]$. 

We have shown 
\begin{equation}
\sum_j\left(\frac{d}{dt}\bigg|_{t=r}T_{ij}(x,t)\right)^2\leq C_1C_2^2C_3(R)^2||u_t-v_s||_{C^0(M,\mathbb{R}^q)}^2
\end{equation}
for all $(x,t)$. Combining this with \eqref{difference of Dirac operaor} and \eqref{difference of T}, we complete the proof.
\end{proof}

The other one is for the parallel transport.
\begin{lem}\label{PT}
Choose $\epsilon$, $\delta$ and $R$ as in \eqref{choice of constants}. If $\epsilon>0$ is small enough, then there exists $C=C(\epsilon)>0$ such that 
\begin{equation}
||P^{v_s,u_0}P^{u_t,v_s}P^{u_0,u_t}Z-Z||\leq C||u_t-v_s||_{C^0(M,\mathbb{R}^q)}||Z||
\end{equation}
for all $Z\in T_{u_0(x)}N$, $u,v\in B^T_R(\bar{u}_0)$, $x\in M$ and $t,s\in[0,T]$.
\end{lem}

Consequently, we also have
\begin{lem}\label{W1p norm}
Choose $\epsilon$, $\delta$ and $R$ as in \eqref{choice of constants}. For $u,v\in B^T_R(\bar{u}_0)$, $s,t\in[0,T]$, the operator norm of the isomorphism of Banach spaces
\begin{equation}
P^{v_s,u_t}: \Gamma_{W^{1,p}}(\Sigma M\otimes(\pi\circ{v_s})^*TN)\to\Gamma_{W^{1,p}}(\Sigma M\otimes(\pi\circ{u_t})^*TN)
\end{equation}
is uniformly bounded, i.e. there exists $C=C(R,p)$ such that
\begin{equation}
||P^{v_s,u_t}||_{L(W^{1,p},W^{1,p})}\leq C
\end{equation}
for all $u,v\in B^T_R(\bar{u}_0)$, $x\in M$ and $t,s\in[0,T]$.
\end{lem}

The proofs of these two lemmas only depend on the existence of the unique shortest geodesic between any two maps in $B^T_R(\bar{u}_0)$, which was already shown in \eqref{shortest geodesic}. Therefore, we omit the detailed proof here. Besides, by Lemma \ref{Dirac along maps u}, one can immediately prove the following Lemma by the Min-Max principle as in \cite{wittmann2017short}.

\begin{lem}\label{dim of kernel u}
Assume that ${\rm dim}_{\mathbb{K}}{\rm ker}(\slashed{D}^{u_0})=2l-1$, where $l\in\mathbb{N}$ and 
\begin{equation}
		\mathbb{K}=\begin{cases}
		\mathbb{C},& \text{if} \ m=0,1({\rm mod} \ 8), \\
		\mathbb{H},& \text{if} \ m=2,4({\rm mod} \ 8).
		\end{cases}
	\end{equation}
Choose $\epsilon$, $\delta$ and $R$ as in Lemma \ref{Dirac along maps u}. If $R$ is small enough, then 
\begin{equation}
{\rm dim}_{\mathbb{K}}{\rm ker}(\slashed{D}^{\pi\circ u_t})=1
\end{equation}
and there exists $\Lambda=\frac12\Lambda(u_0)$ such that 
\begin{equation}
\#\{{\rm spec}(\slashed{D}^{\pi\circ u_t})\cap[-\Lambda,\Lambda]\}=1
\end{equation}
for any $u\in B^T_R(\bar{u}_0)$ and $t\in[0,T]$, where $\Lambda(u_0)$ is a constant such that ${\rm spec}(\slashed{D}^{u_0})\setminus\{0\}\subset\mathbb{R}\setminus(-\Lambda(u_0),\Lambda(u_0))$.
\end{lem}

Once we have the minimality of the kernel in Lemma \ref{dim of kernel u}, we can prove the following uniform bounds for the resolvents, which are  important for the Lipschitz continuity of the solution to the Dirac equation.
\begin{lem}\label{resolvents}
Assume we are in the situation of Lemma \ref{dim of kernel u}. We consider the resolvent $R(\lambda, \slashed{D}^{\pi\circ{u_t}}): \Gamma_{L^2}\to\Gamma_{L^2}$ of $\slashed{D}^{\pi\circ{u_t}}:\Gamma_{W^{1,2}}\to\Gamma_{L^2}$. By the $L^p$ estimate (see Lemma 2.1 in \cite{wittmann2017short}), we know the restriction
\begin{equation}
R(\lambda, \slashed{D}^{\pi\circ{u_t}}): \Gamma_{L^p}\to\Gamma_{W^{1,p}}
\end{equation}
is well-defined and bounded for any $2\leq p<\infty$. If $R>0$ is small enough, then there exists $C=C(p,R)>0$ such that
\begin{equation}
\sup_{|\lambda|=\frac{\Lambda}{2}}||R(\lambda, \slashed{D}^{\pi\circ{u_t}})||_{L(L^p,W^{1,p})}<C
\end{equation}
for any $u\in B^T_R(\bar{u}_0)$, $t\in[0,T]$.
\end{lem}

Now, by the projector of the Dirac operator, we can construct a solution to the constraint equation whose nontrivialness follows from the following lemma.
\begin{lem}\label{projection to kernel}
In the situation of Lemma \ref{dim of kernel u}, for any fixed $u\in B^T_R(\bar{u}_0)$ and any $\psi\in{\rm ker}(\slashed{D}^{u_0})$ with $\|\psi\|_{L^2}=1$, we have
\begin{equation}\label{lower bound u}
\sqrt{\frac12}\leq\|\tilde\psi_1^{u_t}\|_{L^2}\leq1,
\end{equation} 
where $\tilde\psi^{u_t}=P^{u_0,u_t}\psi=\tilde\psi_1^{u_t}+\tilde\psi_2^{u_t}$ with respect to the decomposition $\Gamma_{L^2}={\rm ker}(\slashed{D}^{\pi\circ u_t})\oplus({\rm ker}(\slashed{D}^{\pi\circ u_t}))^\bot$ 
\end{lem}

In Section 3, to show the short-time existence of the heat for $\alpha$-Dirac-harmonic maps, we need the following Lipschitz estimate.

\begin{lem}\label{lip}
Choose $\delta$ as in \eqref{choice of constants}, $\epsilon$ as in Lemma \ref{Dirac along maps u} and Lemma \ref{PT}, $R$ as in Lemma \ref{dim of kernel u} and Lemma \ref{resolvents}. For any harmonic spinor $\psi\in{\rm ker}(\slashed{D}^{u_0})$, we define
\begin{equation}
\bar\psi(u_t):=\tilde\psi^{u_t}_1=-\frac{1}{2\pi i}\int_{\gamma}R(\lambda,\slashed{D}^{\pi\circ u_t})\sigma(u_t)d\lambda
\end{equation}
for any $u\in B^T_R(\bar{u}_0)$, where $\gamma$ is defined in the Section $2$ with $\Lambda=\frac12\Lambda(u_0)$. In particular, $\bar\psi(u_t)\in{\rm ker}(\slashed{D}^{\pi\circ u_t})\subset\Gamma_{C^0}(\Sigma M\otimes(\pi\circ u_t)^*TN)$. We write
\begin{equation}\label{def of psi u}
\psi(u_t):=\psi(u(\cdot,t))=\frac{\bar\psi(u_t)}{\|\bar\psi(u_t)\|_{L^2}}.
\end{equation}
Let $\psi^A(u_t)$ be the sections of $\Sigma M$ such that 
\begin{equation}
\psi(u_t)=\psi^A(u_t)\otimes(\partial_A\circ\pi\circ u_t)
\end{equation}
for $A=1,\cdots,q$. Then there exists $C=C(R,\epsilon,\psi_0)>0$ such that
\begin{equation}\label{difference of psi bar after pt}
\|P^{u_t,v_s}\bar\psi(u_t)(x)-\bar\psi(u_t)(x)\|\leq C\|u_t-v_s\|_{C^0(M,\mathbb{R}^q)}
\end{equation}
and
\begin{equation}\label{psi lip}
\|\psi^A(u_t)(x)-\psi^A(v_s)(x)\|\leq C\|u_t-v_s\|_{C^0(M,\mathbb{R}^q)}
\end{equation}
for all $u,v\in B^T_R(\bar{u}_0)$, $A=1,\cdots,q$, $x\in M$ and $s,t\in[0,T]$.
\end{lem}

\begin{proof}
Using the following resolvent identity for two operators $D_1,D_2$
\begin{equation}
R(\lambda,D_1)-R(\lambda,D_2)=R(\lambda,D_1)\circ(D_1-D_2)\circ R(\lambda,D_2),
\end{equation}
we have
\begin{equation}
\begin{split}
&P^{u_t,v_s}\bar\psi(u_t)-\bar\psi(u_t)\\
&=-\frac{1}{2\pi i}\bigg(\int_{\gamma}R(\lambda,P^{u_t,v_s}\slashed{D}^{\pi\circ u_t}(P^{u_t,v_s})^{-1})P^{u_t,v_s}P^{u_0,u_t}\psi_0\\
&\quad-\int_{\gamma}R(\lambda,\slashed{D}^{\pi\circ v_s})P^{u_0,v_s}\psi_0\bigg)\\
&=-\frac{1}{2\pi i}\int_{\gamma}R(\lambda,P^{u_t,v_s}\slashed{D}^{\pi\circ u_t}(P^{u_t,v_s})^{-1})\bigg(P^{u_t,v_s}P^{u_0,u_t}\psi_0-P^{u_0,v_s}\psi_0\bigg)\\
&\quad-\frac{1}{2\pi i}\int_{\gamma}\bigg(R(\lambda,P^{u_t,v_s}\slashed{D}^{\pi\circ u_t}(P^{u_t,v_s})^{-1})-R(\lambda,\slashed{D}^{\pi\circ v_s})\bigg)P^{u_0,v_s}\psi_0\\
&=-\frac{1}{2\pi i}\int_{\gamma}R(\lambda,P^{u_t,v_s}\slashed{D}^{\pi\circ u_t}(P^{u_t,v_s})^{-1})\bigg(P^{u_t,v_s}P^{u_0,u_t}\psi_0-P^{u_0,v_s}\psi_0\bigg)\\
&\begin{split}
\quad-\frac{1}{2\pi i}\int_{\gamma}&\bigg(R(\lambda,P^{u_t,v_s}\slashed{D}^{\pi\circ u_t}(P^{u_t,v_s})^{-1})\circ\left(P^{u_t,v_s}\slashed{D}^{\pi\circ u_t}(P^{u_t,v_s})^{-1}-\slashed{D}^{\pi\circ v_s}\right)\circ\\
&\quad R(\lambda,\slashed{D}^{\pi\circ v_s})\bigg)P^{u_0,v_s}\psi_0,
\end{split}
\end{split}\end{equation}
where $\gamma$ is defined in \eqref{general curve} with $\Lambda=\frac12\Lambda(u_0)$. Therefore, for $p$ large enough, we get
\begin{equation}
\begin{split}
&||P^{u_t,v_s}\bar\psi(u_t)(x)-\bar\psi(u_t)(x)||\leq C_1||P^{u_t,v_s}\bar\psi^{u_t}-\bar\psi^{v_s}||_{W^{1,p}(M)}\\
&\leq C_2\bigg\|\int_{\gamma}R(\lambda,P^{u_t,v_s}\slashed{D}^{\pi\circ u_t}(P^{u_t,v_s})^{-1})\bigg(P^{u_t,v_s}P^{u_0,u_t}\psi_0-P^{u_0,v_s}\psi_0\bigg)\bigg\|_{W^{1,p}(M)}\\
&\begin{split}
+C_2\bigg\|\int_{\gamma}&\bigg(R(\lambda,P^{u_t,v_s}\slashed{D}^{\pi\circ u_t}(P^{u_t,v_s})^{-1})\circ\left(P^{u_t,v_s}\slashed{D}^{\pi\circ u_t}(P^{u_t,v_s})^{-1}-\slashed{D}^{\pi\circ v_s}\right)\circ\\
&\quad R(\lambda,\slashed{D}^{\pi\circ v_s})\bigg)P^{u_0,v_s}\psi_0\bigg\|_{W^{1,p}(M)}
\end{split}\\
&\leq C_2\int_{\gamma}\bigg\|R(\lambda,P^{u_t,v_s}\slashed{D}^{\pi\circ u_t}(P^{u_t,v_s})^{-1})\bigg(P^{u_t,v_s}P^{u_0,u_t}\psi_0-P^{u_0,v_s}\psi_0\bigg)\bigg\|_{W^{1,p}(M)}\\
&\begin{split}
+C_2\int_{\gamma}&\bigg\|\bigg(R(\lambda,P^{u_t,v_s}\slashed{D}^{\pi\circ u_t}(P^{u_t,v_s})^{-1})\circ\left(P^{u_t,v_s}\slashed{D}^{\pi\circ u_t}(P^{u_t,v_s})^{-1}-\slashed{D}^{\pi\circ v_s}\right)\circ\\
&\quad R(\lambda,\slashed{D}^{\pi\circ v_s})\bigg)P^{u_0,v_s}\psi_0\bigg\|_{W^{1,p}(M)}
\end{split}\\
&\leq C_3\sup\limits_{{\rm Im}(\gamma)}\|R(\lambda,P^{u_t,v_s}\slashed{D}^{\pi\circ u_t}(P^{u_t,v_s})^{-1})\|_{L(L^p,W^{1,p})}\|P^{u_t,v_s}P^{u_0,u_t}\psi_0-P^{u_0,v_s}\psi_0\|_{L^p}\\
&\quad+C_3\sup\limits_{{\rm Im}(\gamma)}\|R(\lambda,P^{u_t,v_s}\slashed{D}^{\pi\circ u_t}(P^{u_t,v_s})^{-1})\|_{L(L^p,W^{1,p})}\sup\limits_{{\rm Im}(\gamma)}\|R(\lambda,\slashed{D}^{\pi\circ v_s})\|_{L(L^p,W^{1,p})}\\
&\quad\quad\|P^{u_t,v_s}\slashed{D}^{\pi\circ u_t}(P^{u_t,v_s})^{-1}-\slashed{D}^{\pi\circ v_s}\|_{L(W^{1,p},L^p)}\|P^{u_0,v_s}\psi_0\|_{L^p}.
\end{split}
\end{equation}

Now, we estimate all the terms in the right-hand side of the inequality above. First, by Lemma \ref{resolvents} and Lemma \ref{W1p norm}, we know that all the resolvents above are uniformly bounded. Next, by Lemma \ref{Dirac along maps u}, we have
\begin{equation}
\|P^{u_t,v_s}\slashed{D}^{\pi\circ u_t}(P^{u_t,v_s})^{-1}-\slashed{D}^{\pi\circ v_s}\|_{L(W^{1,p},L^p)}\leq C(R)\|u_t-v_s\|_{C^0(M,\mathbb{R}^q)}.
\end{equation}
Finally, by Lemma \ref{PT}, we obtain 
\begin{equation}
\|P^{u_t,v_s}P^{u_0,u_t}\psi_0-P^{u_0,v_s}\psi_0\|_{L^p}\leq C(\epsilon,\psi_0)\|u_t-v_s\|_{C^0(M,\mathbb{R}^q)}.
\end{equation}
Putting these together, we get \eqref{difference of psi bar after pt}.

Next, we want to show the following estimate which is very close to \eqref{psi lip}.
\begin{equation}\label{psi bar lip}
\|\bar{\psi}^A(u_t)(x)-\bar{\psi}^A(v_s)(x)\|\leq C(R,\epsilon,\psi_0)\|u_t-v_s\|_{C^0(M,\mathbb{R}^q)}.
\end{equation}

In fact, we have 
\begin{equation*}
\begin{split}
&\|\bar{\psi}^A(u_t)(x)-\bar{\psi}^A(v_s)(x)\|\\
&\leq\|\bar{\psi}(u_t)(x)-\bar{\psi}(v_s)(x)\|_{\Sigma_xM\otimes\mathbb{R}^q}\\
&\leq\|P^{u_t,v_s}\bar{\psi}(u_t)(x)-\bar{\psi}(v_s)(x)\|_{\Sigma_xM\otimes\mathbb{R}^q}+\|P^{u_t,v_s}\bar{\psi}(u_t)(x)-\bar{\psi}(u_t)(x)\|_{\Sigma_xM\otimes\mathbb{R}^q}\\
&=\|P^{u_t,v_s}\bar{\psi}(u_t)(x)-\bar{\psi}(v_s)(x)\|_{\Sigma_xM\otimes T_{(\pi\circ v_s(x))}N}+\|P^{u_t,v_s}\bar{\psi}(u_t)(x)-\bar{\psi}(u_t)(x)\|_{\Sigma_xM\otimes\mathbb{R}^q}\\
&\leq C(R,\epsilon,\psi_0)\|u_t-v_s\|_{C^0(M,\mathbb{R}^q)}+\|P^{u_t,v_s}\bar{\psi}(u_t)(x)-\bar{\psi}(u_t)(x)\|_{\Sigma_xM\otimes\mathbb{R}^q}.
\end{split}\end{equation*}
It remains to estimate the last term in the inequality above. To that end, let $\gamma(r):=\exp_{(\pi\circ u_t)(x)}(r\exp^{-1}_{(\pi\circ u_t)(x)}(\pi\circ u_t(x)))$, $r\in[0,1]$, be the unique shortest geodesic of $N$ from $(\pi\circ u_t)(x)$ to $(\pi\circ v_s)(x)$. Let $X\in T_{\gamma(0)}N$ be given and denote by $X(r)$ the unique parallel vector field along $\gamma$ with $X(0)=X$. Then we have
\begin{equation}
P^{u_t,v_s}X-X=X(1)-X(0)=\int_0^1\frac{dX}{dr}\bigg|_{r=\xi}d\xi=\int_0^1II(\gamma'(r),X(r))dr.
\end{equation}
Therefore, 
\begin{equation}
\|P^{u_t,v_s}X-X\|_{\mathbb{R}^q}\leq C_1\sup\limits_{r\in[0,1]}\|\gamma'(r)\|_{N}\sup\limits_{r\in[0,1]}\|X(r)\|_{N}=C_1\|\gamma'(0)\|_{N}\|X\|_{N}
\end{equation}
where $II$ is the second fundamental form of $N$ in $\mathbb{R}^q$ and $C_1$ only depends on $N$. Using \eqref{control distance on N} and the Lipschitz continuity of $\pi$ we get
\begin{equation}
\|\gamma'(0)\|_{N}\leq d^N((\pi\circ u_t)(x),(\pi\circ v_s)(x))\leq C_2\|u_t(x)-v_s(x)|\|_{\mathbb{R}^q}
\end{equation}
and 
\begin{equation}
\|P^{u_t,v_s}X-X\|_{\mathbb{R}^q}\leq C_3\|u_t(x)-v_s(x)|\|_{\mathbb{R}^q}\|X\|_{N}.
\end{equation}
This implies
\begin{equation}
\|P^{u_t,v_s}\bar{\psi}(u_t)(x)-\bar{\psi}(u_t)(x)\|_{\Sigma_xM\otimes\mathbb{R}^q}\leq C(R,\epsilon,\psi_0)\|u_t(x)-v_s(x)|\|_{\mathbb{R}^q}.
\end{equation}
Hence, \eqref{psi bar lip} holds.

Now, using \eqref{difference of psi bar after pt} and \eqref{psi bar lip}, we get
\begin{equation*}
\begin{split}
&\|\psi^A(u_t)(x)-\psi^A(v_s)(x)\|
=\bigg\|\frac{\bar\psi^A(u_t)(x)}{\|\bar\psi(u_t)\|_{L^2}}-\frac{\bar\psi^A(u_t)(x)}{\|\bar\psi(v_s)\|_{L^2}}+\frac{\bar\psi^A(u_t)(x)}{\|\bar\psi(v_s)\|_{L^2}}-\frac{\bar\psi^A(v_s)(x)}{\|\bar\psi(v_s)\|_{L^2}}\bigg\|\\
&\leq\frac{\bar\psi^A(u_t)(x)}{\|\bar\psi(u_t)\|_{L^2}\|\bar\psi(v_s)\|_{L^2}}\bigg|\|\bar\psi(v_s)\|_{L^2}-\|\bar\psi(u_t)\|_{L^2}\bigg|+\frac{1}{\|\bar\psi(v_s)\|_{L^2}}\|\bar\psi^A(u_t)(x)-\bar\psi^A(v_s)(x)\|\\
&=\frac{\bar\psi^A(u_t)(x)}{\|\bar\psi(u_t)\|_{L^2}\|\bar\psi(v_s)\|_{L^2}}\bigg|\|\bar\psi(v_s)\|_{L^2}-\|P^{u_t,v_s}\bar\psi(u_t)\|_{L^2}\bigg|\\
&\quad+\frac{1}{\|\bar\psi(v_s)\|_{L^2}}\|\bar\psi^A(u_t)(x)-\bar\psi^A(v_s)(x)\|\\
&\leq\frac{\bar\psi^A(u_t)(x)}{\|\bar\psi(u_t)\|_{L^2}\|\bar\psi(v_s)\|_{L^2}}\|P^{u_t,v_s}\bar\psi(u_t)-\bar\psi(v_s)\|_{L^2}+\frac{1}{\|\bar\psi(v_s)\|_{L^2}}\|\bar\psi^A(u_t)(x)-\bar\psi^A(v_s)(x)\|\\
&\leq\bigg(\frac{\bar\psi^A(u_t)(x)}{\|\bar\psi(u_t)\|_{L^2}\|\bar\psi(v_s)\|_{L^2}}+\frac{1}{\|\bar\psi(v_s)\|_{L^2}}\bigg)C(R,\epsilon,\psi_0)\|u_t-v_s\|_{C^0(M,\mathbb{R}^q)}.
\end{split}
\end{equation*}
Then the inequality \eqref{psi lip} follows from Lemma \ref{projection to kernel} and \eqref{psi bar lip}. This completes the proof.

\end{proof}


\nocite{*}


\bibliographystyle{amsplain}
\bibliography{reference} 
 
\end{document}